\documentclass{ws-ijbc}

\usepackage{ws-rotating} 
\usepackage{xcolor}
\usepackage[verbose]{hyperref}
\hypersetup{colorlinks=false,allbordercolors=blue,pdfborderstyle={/S/U/W 1}}
\usepackage{pgfplots} 
  \pgfplotsset{compat=1.18}
\usepackage{tikz}
 \usetikzlibrary{arrows.meta}

\begin{document}


\markboth{J. L. Bravo et al.}
{Piecewise linear differential equations with three zones}

\title{Upper bound for the number of limit cycles for symmetric
piecewise linear differential equations with three zones}

\author{Jos\'e Luis Bravo}
\address{Departamento de Matem\'aticas, Universidad de Extremadura\\
Avenida de Elvas s/n, 06006 Badajoz, Spain\\
trinidad@unex.es}

\author{Victoriano Carmona\footnote{Corresponding author.}}
\address{Departamento de Matem\'atica Aplicada II \& IMUS,
Universidad de Sevilla\\
Escuela Polit\'ecnica Superior, Calle Virgen de \'Africa 7\\
41011 Sevilla, Spain\\
vcarmona@us.es}

\author{Manuel Fern\'andez}
\address{Instituto de Matem\'aticas, Universidad de Extremadura\\
Avenida de Elvas s/n, 06006 Badajoz, Spain\\
ghierro@unex.es}

\author{Ignacio Ojeda}
\address{Departamento de Matem\'aticas, Universidad de Extremadura\\
Avenida de Elvas s/n, 06006 Badajoz, Spain\\
ojedamc@unex.es}

\maketitle


\begin{abstract}
The study of the periodic dynamics of the family of continuous observable and non-controllable three-dimensional symmetric piecewise linear systems with three zones considered here can be reduced to the study of periodic solutions of the piecewise differential equation $x'=ax+(b-a)\mathop{\rm sat}(x)+\mu\sin t$, where $\operatorname{sat}$ stands for the normalized saturation function.

This paper proves that the number of limit cycles (that is, isolated periodic solutions) of the equation is uniformly bounded independently of $a,b,\mu$. Moreover, when $ab>0$, the equation has exactly one limit cycle for every $\mu>0$. When $ab<0$, it has exactly three limit cycles for sufficiently small $\mu>0$ and exactly one for sufficiently large $\mu$.

When $ab<0$, using Melnikov theory for a deformation of the equation, we prove that, for each fixed $\mu>0$ away from two critical values and for sufficiently small nonzero values of the perturbation parameter, there are exactly three, five or one limit cycles, depending on $\mu$. This strengthens the conjecture that the equation has at most five limit cycles.
\end{abstract}


\section{Introduction and Main Results}

Within the framework of temporal evolution systems and the study of nonlinear oscillators, non-smooth differential systems naturally emerge, and, as a particular case, piecewise linear dynamic systems. These latter systems, beyond being mathematically attractive structures in themselves, allow the modeling of different electronic \cite{Carmona20053153,CarmonaEtAl02,Freire1999895} or mechanical devices \cite{PogorilyiTrivailoJazar+2014+189+196,6371308} and are capable of capturing the nonlinear dynamics of the differentiable systems \cite{Arneodo1982171,Carmona2010,Carmona20081032, Tresser84}. Furthermore, when continuity conditions are not required, they present new phenomena that seem well-suited to explain behaviors intrinsic to various physical mechanisms \cite{doi:10.1137/17M1110328,Bernardo2008,Leine2004} and even neurological mechanisms \cite{10.1063/5.0101778}.
	
The study of the dynamics of non-smooth systems begins with identifying the regions where the system behaves in a differentiable manner and describing the boundary zones between them. For $n$-dimensional systems, the boundary separations are often parallel hyperplanes, and the inherent problem of the dynamical behavior of the system shifts to understanding how points in a boundary separation transform, through the flow of the system, into points of an adjacent zone (or even the same zone). The map that transforms points from one boundary zone to another is typically called a transition map, and when the starting and ending zones are the same, it is usually referred to as a Poincaré half-map. The global behavior of the system is understood by knowing the conduct of these maps, their composition, and the behavior in the boundary separations. When the vector field of the system is continuous, or the separating boundaries consist of crossing regions, the evolution across these boundaries is unambiguous. However, otherwise, it is common to resort to Filippov's convention \cite{Filippov88} to establish the dynamical behavior in the boundary zones, where sliding or escaping regions frequently appear.
	
The simplest piecewise linear system is the one constructed in two dimensions with two linearity regions separated by a straight line. When the continuity is imposed on the vector field that defines it, or when there is no sliding region on the separation line, the dynamical behavior is almost completely understood thanks to the novel integral characterization of the Poincaré half-maps recently discovered \cite{CARMONA2021319}. For instance, it is already known that the system has at most one limit cycle, that is, a unique isolated periodic orbit \cite{CarmonaEtAl19b,Caretalpre22}, and it is conjectured that it has at most one non-trivial isolated invariant curve. However, when the presence of the sliding region on the separating line is allowed, the dynamics are not understood, and until very recently, it was unknown whether the system could have an infinite number of limit cycles. It is now established that the maximum number of limit cycles is finite, and an explicit upper bound for this number is provided in \cite{Carmona2022}.
	
For the $n$-dimensional case, the study of a parametrized family of autonomous piecewise linear systems begins with their formulation in an appropriate canonical form \cite{Kahlert1990373,Kahlert1992222}. The commonly established canonical form is the well-known Liénard canonical form \cite{CarmonaEtAl02, FreireEtAl12}, which has proven to be highly useful for systems with two zones of linearity and symmetric systems with three zones of linearity. Systems that cannot be written in Liénard canonical form (commonly referred to in the literature as non-observable systems in control theory \cite{barnett1985introduction}) can be decoupled, reducing the dimensionality of the autonomous piecewise linear dynamics \cite{CarmonaEtAl02}. However, it is possible to find observable systems for which the dimensionality of the piecewise linear problem can be reduced. In such cases, the lower-dimensional system ceases to be autonomous, resulting in a non-autonomous piecewise linear system \cite{CarmonaEtAl02,Carmona200471,CFPT2}. It is commonly said that the original autonomous system is observable but not controllable. 
	
Note that dynamic behavior in three-dimensional autonomous systems is far from being fully understood because, among other factors, the emergence of complex dynamics featuring strange attractors and chaotic behavior dictates that its study must be addressed from distinct angles, with different strategies, and from various perspectives. One of the strategies followed relies on the detection of singular or specific situations that act as generators and organizing centers for the dynamics, particularly for complex and/or chaotic dynamics. In these specific situations, researchers often seek the degeneracy of equilibrium points, the presence of global connections or the emergence of a continuum of periodic orbits within a given invariant manifold. 
	
Obviously, three-dimensional piecewise linear systems are not exempt, and although the word ``linear'' might suggest simple dynamics, the piecewise character drives the system away from simple behavior. As a case in point, we bring up the work of T. Matsumoto \cite{matsumoto} on chaotic behavior in the well-known Chua's circuit, which is modeled by a symmetric piecewise linear system with three zones of linearity. Another interesting work in this regard is that by Carmona et al. \cite{Carmona200471} on the study of dynamics in a specific or singular family of three-dimensional piecewise linear systems with two zones of linearity that possess an invariant manifold foliated with periodic orbits. The study of the periodic behavior in this family is not trivial and presents a strong interest in itself, as its analysis is reduced to the study of the limit cycles of a one-dimensional piecewise linear equation with two zones periodically forced, which is reminiscent of quadratic Abel-type or Riccati equations. The exhaustive bifurcation analysis developed on the one-dimensional equation demonstrates that the invariant manifolds can be bounded, which, along with the numerical simulations included, shows how the manifolds act as a seed for richer and more complex dynamics, where invariant tori, periodic orbits of any period, and even heteroclinic cycles and homoclinic orbits can be observed when the system is perturbed. 
	
The analysis framed in this work moves in that direction: it focuses on the study of the periodic dynamics of a family of symmetric three-dimensional piecewise linear systems with three zones of linearity that possess invariant manifolds foliated with periodic orbits, that is, the continuous observable and non-controllable three-dimensional symmetric piecewise linear systems with three zones of linearity. 
	
As it has been proved in~\cite[Lemma 7 (b)]{CarmonaEtAl02}, a continuous three-dimensional symmetric piecewise linear system with three zones of linearity can be written in the form 
\begin{equation}\label{ecu:sispwl3}
	\dot{\mathbf{x}}=A\mathbf{x}+\mathbf{b}\mathop{\rm sat}
		\left(\mathbf{e}_1^T \mathbf{x}\right),
\end{equation}
where $\mathbf{x}=(x_1,x_2,x_3)^T\in\mathbb{R}^3,$ $A$ is a $3\times 3$ real matrix, $\mathbf{b}\in\mathbb{R}^3$, $\mathbf{e}_1=(1,0,0)^T$ and $\mathop{\rm sat}(\cdot)$ is the normalized saturation function 
\begin{equation}
	\label{ecu:sat}
	\mathop{\rm sat}(x)=
    \left\{\begin{array}{ccr} 
        x & {\rm if} & |x|<1, \\ \noalign{\smallskip}
		\mathop{\rm sgn}(x) & {\rm if} & |x|\geq 1,
	\end{array}\right.
\end{equation}
where $\mathop{\rm sgn}(\cdot)$ denotes the sign function and the dot denotes differentiation with respect to time $s$.
	
From Proposition 16 of \cite{CarmonaEtAl02} and using the change of variable $x_2\to -x_2$, it follows that system \eqref{ecu:sispwl3} can be transformed into the form 
\begin{equation}
	\dot{\mathbf{x}}=B \mathbf{x}+ \mathbf{c}\mathop{\rm sat}\left(\mathbf{e}_1^T \mathbf{x}\right),\label{ecu'fcanoobsethree'}
\end{equation}
with $\mathbf{c}\in\mathbb{R}^3$ and 
\[
    B=\left(\begin{array}{crr}
	       T & -1 & 0 \\ M & 0 & -1 \\ D & 0 & 0
    \end{array}\right),
\]
provided that the observability matrix 
${\mathcal{O}}=\left( 
\begin{array}{c|c|c}
		{\mathbf{e}}_{1} & A^{T}{\mathbf{e}}_{1} & (A^{T})^{2}{\mathbf{e}}_{1}
\end{array}
\right) ^{T}$ has full rank. Here, $T,-M$ and $D$ are the coefficients of the characteristic polynomial of the similar matrices $A$ and $B$, that is, $p_A(\lambda)=p_B(\lambda)=\det(B-\lambda I)=-\lambda^3+T\lambda^2-M\lambda+D$. Moreover, by Proposition~13 of \cite{CarmonaEtAl02}, system \eqref{ecu'fcanoobsethree'} is not controllable if and only if the matrices $B$ and $B+\mathbf{c}\mathbf{e}_1^T$ share some eigenvalue. When the shared eigenvalue is real, the study of nonlinear oscillations for the system can be done, as stated in Proposition 20 of \cite{CarmonaEtAl02}, using the analysis of periodic orbits of an autonomous two-dimensional symmetric piecewise linear system. If the shared eigenvalue is $\alpha+\rm{i}\beta$, with $\beta>0$, then according to Proposition 21 of \cite{CarmonaEtAl02} system \eqref{ecu'fcanoobsethree'} can be transformed into the form 
	\begin{equation}
		\label{ecu:sisttriobsnocont}
		\left\{
		\begin{array}{l}
			\dot x_1=(T-2\alpha)x_1+(\widetilde T-T)\mathop{\rm sat}(x_1)-x_2, \\
			\noalign{\smallskip}
			\dot x_2=2\alpha x_2+x_3, \\
			\noalign{\smallskip}
			\dot x_3=-\left(\alpha^2+\beta^2 \right)x_2,
		\end{array}
		\right.
	\end{equation}
	where $T$ and $\widetilde T$ are the traces of the matrices $B$ and $B+\mathbf{c}\mathbf{e}_1^T$, respectively. 
	
	If system \eqref{ecu:sisttriobsnocont} has a non-trivial periodic orbit, then $\alpha=0$ and each cylinder of the family $\beta^2\left(\beta^2x_2^2+x_3^2\right)=r^2$, with $r > 0$, is an invariant surface for the system. Moreover, when $\alpha=0$, the change of variables and parameters 
	\[
	x=x_1, \, y=x_2/\beta, \, z=x_3/\beta^2, \, t=\beta s,\,  a=T/\beta,\,  b=\widetilde T/\beta
	\]
	allows us to write system \eqref{ecu:sisttriobsnocont} as
	\begin{equation}
		\label{ecu:sisttriobsnocontb}
		\left\{
		\begin{array}{l}
			x'=ax+(b-a)\mathop{\rm sat}(x)-y, \\
			\noalign{\smallskip}
			y'=z, \\
			\noalign{\smallskip}
			z'=-y,
		\end{array}
		\right.
	\end{equation}
	and the invariant cylinders are transformed into $y^2+z^2=\mu^2$, with $\mu=r/\beta^3>0$. Here, the prime denotes the derivative with respect to the variable $t$.  
    
    The main results are collected in Theorems~\ref{theo:tridimensional} and  \ref{theo:tridimensional2}.
    
		\begin{theorem}\label{theo:tridimensional}
		Consider the three-dimensional differential system \eqref{ecu:sisttriobsnocontb}.
				\begin{enumerate}	
			\item 		If $ab>0$, system \eqref{ecu:sisttriobsnocontb} has only one equilibrium point and a unique periodic orbit living in each invariant cylinder $y^2+z^2=\mu^2$, with $\mu > 0$ (see Figure \ref{fig:optridipos}). Indeed, this periodic orbit is, for each cylinder, attractive when $a+b<0$ and repulsive when $a+b>0$. 
			\item If $ab<0$,  system \eqref{ecu:sisttriobsnocontb} has exactly three equilibrium points. Moreover, it possesses exactly three periodic orbits living in the invariant cylinder $y^2+z^2=\mu^2$ when  
			\[
			0<\mu \leq \min\left\{ \frac{-b\sqrt{a^2 + 1}}{a},\sqrt{b^2+1} \right\}
			\]
			and it possesses a unique periodic orbit living in the invariant cylinder $y^2+z^2=\mu^2$ when the radius $\mu$ is sufficiently large. 
	\end{enumerate}	
	\end{theorem}

\begin{figure}[h]
	\begin{center}
		\includegraphics[height=30mm]{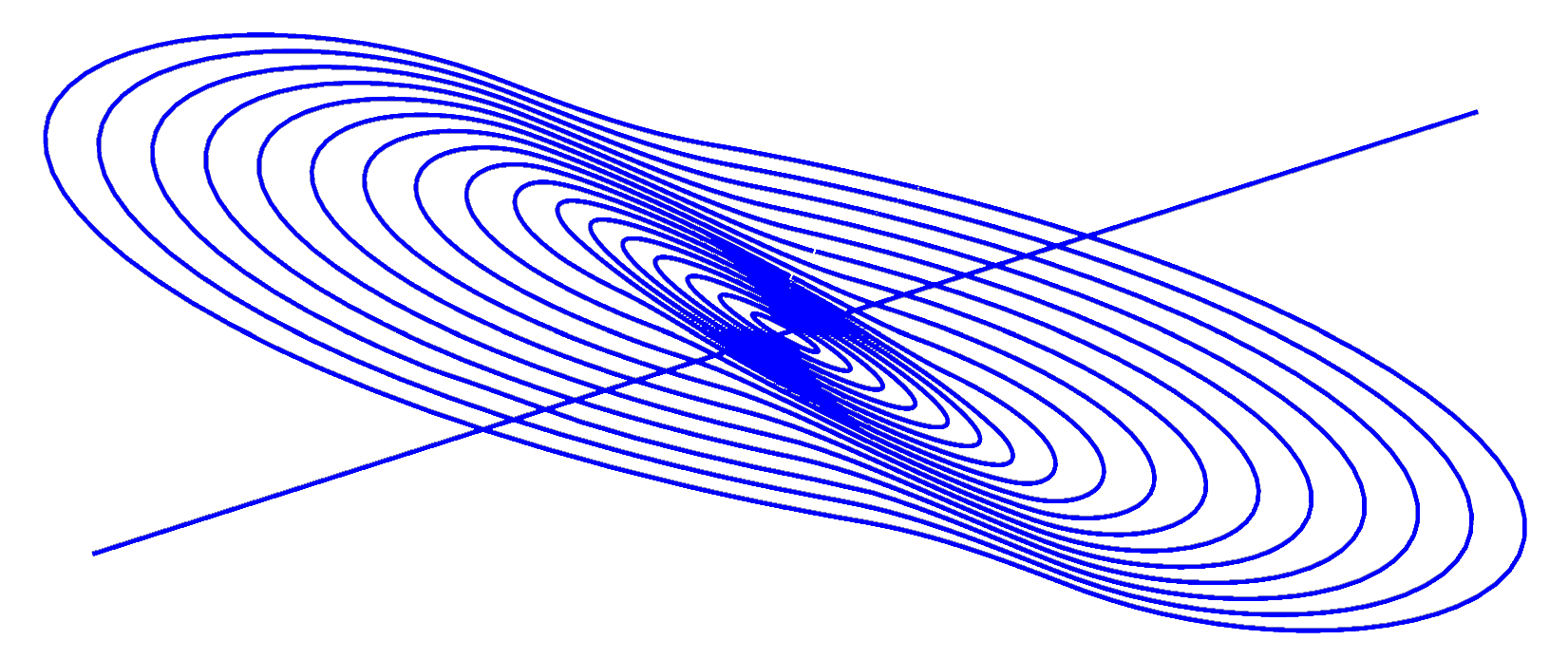}
	\end{center}
	\caption{Structure of periodic orbits of the three-dimensional system \eqref{ecu:sisttriobsnocontb} when $ab>0$.}\label{fig:optridipos}
\end{figure}

Theorem~\ref{theo:tridimensional} completely solves the case $ab>0$ and partially solves the case $ab<0$. To study the case $ab<0$ for small values of the parameters,
we replace $a$ and $b$ in~\eqref{ecu:sisttriobsnocontb}
with $\varepsilon a$ and $\varepsilon b$, respectively,
obtaining the following perturbative system:
	\begin{equation}
		\label{ecu:sisttriobsnocontbpert}
		\left\{
		\begin{array}{l}
			x'=\varepsilon a x +\varepsilon (b- a)\mathop{\rm sat}(x)-y, \\
			\noalign{\smallskip}
			y'=z, \\
			\noalign{\smallskip}
			z'=-y.
		\end{array}
		\right.
	\end{equation}

For~\eqref{ecu:sisttriobsnocontbpert}, we obtain the following result, illustrated in Figures~\ref{fig:optridi1} and \ref{fig:optridi}.

\begin{theorem}\label{theo:tridimensional2} 
Assume $ab<0$, and define
\[
c:=\frac{\pi b}{b-a}, \qquad
\mu_1:=\frac{c}{\sin(c)}, \qquad
\mu_2:=\frac{1}{\cos(c/2)}.
\]
Then $0<\mu_2<\mu_1$, and for each
$\mu>0$ with $\mu\neq\mu_1,\mu_2$, there exists
$\varepsilon_0(\mu)>0$ such that the following
statements hold whenever
$0<|\varepsilon|<\varepsilon_0(\mu)$:
\begin{enumerate}
	\item  if  $0<\mu<\mu_2$, system \eqref{ecu:sisttriobsnocontbpert} possesses exactly three periodic orbits living in the invariant cylinder $y^2+z^2=\mu^2$.
	\item if  $\mu_2<\mu<\mu_1$, system \eqref{ecu:sisttriobsnocontbpert} possesses exactly five periodic orbits living in  the invariant cylinder $y^2+z^2=\mu^2$.
	\item  if  $\mu_1<\mu$, system \eqref{ecu:sisttriobsnocontbpert} possesses exactly one periodic orbit living  in the invariant cylinder $y^2+z^2=\mu^2$.
\end{enumerate}
\end{theorem}

\begin{figure}[htbp]
\centering
\begin{minipage}[t]{0.48\textwidth}
    \vspace{0pt}
    \centering
    \parbox[c][75mm][c]{\linewidth}{
        \centering
        \includegraphics[
            width=\linewidth,
            height=75mm,
            keepaspectratio
        ]{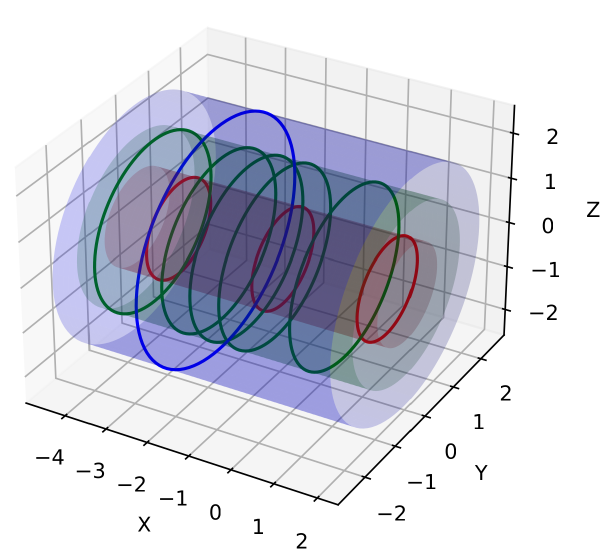}
    }
    \par
    \caption{Structure of periodic orbits of the three-dimensional
    system \eqref{ecu:sisttriobsnocontb} when $ab<0$ and both
    parameters are sufficiently small in absolute value.
    The values of the parameters are $a=0.2$, $b=-0.3$.
    The invariant cylinders depicted have radii $1.0$, $1.8$
    and $2.5$. The periodic solutions have been computed using
    RK45 to integrate the system and \texttt{brentq} to compute
    the zeros of the displacement function.}
    \label{fig:optridi1}
\end{minipage}\hfill
\begin{minipage}[t]{0.48\textwidth}
    \vspace{0pt}
    \centering
    \parbox[c][75mm][c]{\linewidth}{
        \centering
        \includegraphics[
            width=\linewidth,
            height=65mm,
            keepaspectratio
        ]{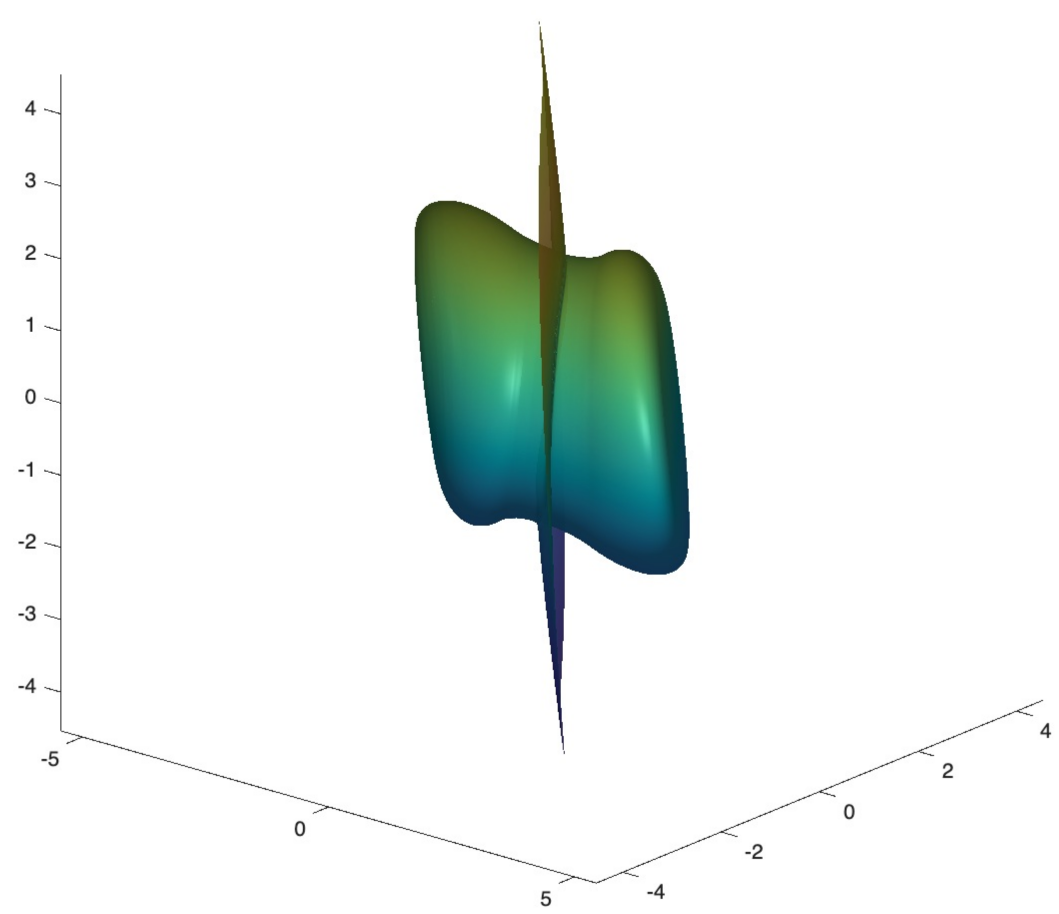}
    }
    \par
    \caption{Numerically computed surface of periodic orbits
    of \eqref{ecu:sisttriobsnocontb} when $ab<0$ and both
    parameters are sufficiently small in absolute value.
    The values of the parameters are $a=0.2$, $b=-0.3$.}
    \label{fig:optridi}
\end{minipage}
\end{figure}
    
In order to prove Theorems~\ref{theo:tridimensional}
and~\ref{theo:tridimensional2}, we analyze the periodic
solutions and bifurcations of the family of differential equations
\begin{equation}
	\label{ecu:main}
	x'=ax+(b-a)\mathop{\rm sat}(x)+\mu\sin t,
\end{equation}
taking the radius of the invariant cylinders as the bifurcation parameter; that is, the parameter $\mu > 0$. For Theorem~\ref{theo:tridimensional2}, we consider the
perturbative family obtained from~\eqref{ecu:main} by
replacing $a$ and $b$ with $\varepsilon a$ and
$\varepsilon b$, respectively.

Note that the $2\pi$-periodic scalar differential equation \eqref{ecu:main} can be written in the form
\begin{equation}
	\label{eq:main1}
	x'= h(t,x) :=f(x)+\mu\sin t,
\end{equation}
with
\begin{equation}\label{eq:f}
	f(x) := ax +\frac{a-b}{2}\big(|x-1|-|x+1|\big) =
	\begin{cases}
		a\, x+(a-b),\quad \text{if }x\leq -1,\\
		b\, x,\quad \text{if }-1<x < 1,\\
		a\, x+(b-a),\quad \text{if } x \geq 1.
	\end{cases}
\end{equation}

We denote by $u(t,\tau,x)$ the maximal solution of the differential equation \eqref{eq:main1} determined by the initial condition $u(\tau,\tau,x)=x$. By Massera's Theorem~\cite{Mss}, periodic solutions of equation \eqref{eq:main1} are always $2 \pi$-periodic solutions. Therefore, periodic means $2\pi$-periodic in the following.

\begin{definition}
Suppose that $u(t,\tau,x)$ is periodic. It is said to be singular or multiple if $u_x(\tau+2\pi,\tau,x)=1$, where the subscript denotes the partial derivative; otherwise, it is said to be simple or hyperbolic. Isolated periodic solutions are also called limit cycles. A continuum of periodic solutions is called a center of equation \eqref{eq:main1}.
\end{definition}

Equation \eqref{eq:main1} is a piecewise linear differential equation with three linearity zones, and its solutions may visit all three linearity zones, only two of them, or just one. If a solution resides in only one of the linearity zones, the dynamical behavior is linear, so at most one limit cycle exists in each linearity zone.
	 
The study of limit cycles becomes more complex when the solutions of the equation occupy more than one linearity zone. For example, as shown in \cite{Carmona200471}, the dynamics of solutions occupying two adjacent zones is non-trivial, and saddle-node bifurcations of periodic solutions may arise. The dynamical behavior of a piecewise linear equation with two zones of linearity is well established in \cite{Carmona200471}. The differential equation \eqref{eq:main1} has also been partially studied in \cite{CFPT2}. In that work, the authors conjecture that, for each fixed pair $a,b$ with $ab<0$, equation~\eqref{eq:main1} has at most five limit cycles for all $\mu>0$, and exactly five for $\mu$ in a nonempty open interval.

Numerical computations support this conjecture. For $a=0.2$ and $b=-0.3$, Figure~\ref{fig:bif} shows in blue a schematic representation of the bifurcation diagram defined by $u(2\pi,0,x,\mu)-x=0$, with the dependence of $u$ on $\mu$ made explicit.

\begin{figure}[h]
\centering
\begin{tikzpicture}
\begin{axis}[
    width=7cm,
    height=5cm,
    scale only axis,
    xmin=0, xmax=3,
    ymin=-4.5, ymax=4.5,
    xlabel={$\mu$},
    ylabel={$x$},
    xtick={0,1,2,3},
    ytick={-4,-2,0,2,4},
    grid=major,
    grid style={gray!55,thin},
    axis line style={black},
    tick align=outside,
    clip=true,
]
\addplot[blue,thick,domain=0:3,samples=2] {-x};

\draw[blue,thick]
    (axis cs:0,2.5)
    -- (axis cs:1.52,0.98)
    .. controls (axis cs:1.78,0.72) and (axis cs:1.98,0.15)
    .. (axis cs:2.00,-0.35)
    .. controls (axis cs:2.01,-0.72) and (axis cs:1.86,-0.92)
    .. (axis cs:1.78,-1.05)
    .. controls (axis cs:1.72,-1.18) and (axis cs:1.72,-1.45)
    .. (axis cs:1.72,-1.68);

\draw[blue,thick]
    (axis cs:1.72,-1.68)
    .. controls (axis cs:1.70,-2.15) and (axis cs:1.84,-2.65)
    .. (axis cs:1.94,-3.03)
    .. controls (axis cs:2.05,-3.42) and (axis cs:2.03,-3.77)
    .. (axis cs:1.93,-3.94)
    .. controls (axis cs:1.82,-4.13) and (axis cs:1.61,-4.08)
    .. (axis cs:1.47,-4.00)
    -- (axis cs:0,-2.5);

\draw[gray!75,dashed,thick]
    (axis cs:1,-4.5) -- (axis cs:1,4.5);
\draw[gray!75,dashed,thick]
    (axis cs:1.8,-4.5) -- (axis cs:1.8,4.5);
\draw[gray!75,dashed,thick]
    (axis cs:2.5,-4.5) -- (axis cs:2.5,4.5);
\end{axis}
\end{tikzpicture}
\caption{Schematic representation of the bifurcation diagram for~\eqref{eq:main1} with $a=0.2$ and $b=-0.3$. The periodic solutions were computed using RK45, with \texttt{brentq} used to locate the zeros of the displacement function~\cite{PythonSoftwareFoundation2023}. From left to right, the vertical dashed lines at $\mu=1$, $\mu=1.8$ and $\mu=2.5$ correspond to three, five and one limit cycles, respectively.}
\label{fig:bif}
\end{figure}
	
The paper is organized as follows: Section~\ref{sect2} presents some general results on the equation \eqref{eq:main1}. Section~\ref{Sect3} is devoted to the study of equation \eqref{eq:main1} when $ab=0$ or $ab>0$. The first case is the only one in which centers (continuum of periodic solutions) can appear, and we detail the possible phase portraits. When $ab>0$, we show that equation \eqref{eq:main1} has a unique periodic solution, while if $a b < 0$, we prove, in Section~\ref{Sect:Finiteness}, that it is possible to give a uniform upper bound (independently of $a,b$ and $\mu$) for the number of limit cycles in equation \eqref{eq:main1} (Theorem \ref{Th:finiteness}). The proof relies on Khovanski\u{\i}'s theorem. 
	
In Section~\ref{Sect:Compacity}, for $ab<0$, we prove that if $\mu>0$ is small enough, equation~\eqref{eq:main1} has exactly three limit cycles (Theorem~\ref{Th:-mu_small}), and if $\mu$ is large enough, equation~\eqref{eq:main1} has only one (Theorem~\ref{Th:-mu_large}). The results of these sections allow us to prove Theorem~\ref{theo:tridimensional}.
	
In Section~\ref{Sect:Melnikov}, using Melnikov theory, we analyze the first-order perturbation
\begin{equation}\label{eq:ef}
    x'=\varepsilon f(x)+\mu\sin t.
\end{equation}
We analytically describe the zero set of its first-order Melnikov function in the $(\mu,x)$-plane (sketched in Figure~\ref{fig:zeroset}). Theorem~\ref{theo:Melnikov} provides partial validation of the conjecture presented in~\cite{CFPT2} and, through the scalar reduction, yields Theorem~\ref{theo:tridimensional2}.
    
\section{Preliminary Results}\label{sect2}

In this section, we show some general results that will be useful later. First, note that the function $h(t,x)$ defined in \eqref{eq:main1} is $2 \pi$-periodic with respect to $t$. Taking into account that  $h(t,x)$ is globally Lipschitz continuous with respect to $x$, the solution $u(t,\tau,x)$ is defined for every $t\in\mathbb{R}$. 

Note that, for $p,q,\mu\in\mathbb{R}$, the solution $u(t,\tau,x)$ of 
\begin{equation}\label{eq:linearode}
    x'=px+q+\mu\sin t,
\end{equation}
with initial condition $u(\tau,\tau,x)=x$ is  
\begin{equation}\label{eq:sollinear}
u(t,\tau,x)= e^{p (t-\tau)}x+\int_{\tau}^t (q+\mu \sin( s)) e^{p(t- s)}\,ds.
\end{equation}
For $t\geq\tau$, since \eqref{eq:main1} is piecewise linear, if $u(s,\tau,x)\leq -1$ for every $s\in[\tau,t]$, then
\begin{equation}\label{f1}
u(t,\tau,x)= e^{a (t-\tau)}x+\int_{\tau}^t (a-b+\mu \sin( s)) e^{a(t- s)}\,ds,
\end{equation}
if $-1\leq u(s,\tau,x)\leq 1$ for every $s\in[\tau,t]$, then
\begin{equation}\label{f2}
u(t,\tau,x)=e^{b( t-\tau)}x+\int_{\tau}^t \mu \sin(s) e^{b(t- s)}\, ds
\end{equation}
and if $u(s,\tau,x)\geq 1$ for every $s\in[\tau,t]$, then
\begin{equation}\label{f3}
u(t,\tau,x)= e^{a( t-\tau)}x+ \int_{\tau}^t (b-a+\mu \sin( s))e^{a(t- s)}\,ds.
\end{equation}

Define the Poincaré map for equation \eqref{eq:main1} as 
\begin{equation}\label{eq:Poincaremap}
P(x):=u(2 \pi,0,x)
\end{equation}
and the displacement function of \eqref{eq:main1} as
\begin{equation}\label{eq:disp_app}
d(x) := P(x)-x.
\end{equation}
Notice that the periodic solutions of equation \eqref{eq:main1} correspond to the fixed points of the Poincaré map $P$; that is, the zeroes of the displacement function $d$. Moreover, the limit cycles of the differential equation \eqref{eq:main1} correspond to the isolated zeroes of $d$.

From the odd symmetry of function $f$ given in \eqref{eq:f} it follows that $-u(t+\pi,\tau,x)$ is a solution of \eqref{eq:main1}. Thus, 
\[-u(t+\pi,0,x)=u(t,0,-u(\pi,0,x)),\]
and if we denote 
\begin{equation}\label{eq:semipoinca}
Q(x):=-u(\pi,0,x),
\end{equation}
then
\[
Q^2(x)=Q(Q(x))=Q(-u(\pi,0,x))=-u(\pi,0,-u(\pi,0,x))=u(2\pi,0,x),
\]
that is, $Q^2(x)=P(x)$. The map $Q$ is called a Poincaré half-map for equation \eqref{eq:main1}.

\begin{definition}
We say that the solution $u(t,0,x)$ is a symmetric periodic solution of \eqref{eq:main1} if $x$ is a fixed point of the map $Q$.
\end{definition}

The following result ensures that the Poincaré map $P$, the displacement function $d$, and the Poincaré half-map $Q$ are functions of class $\mathcal{C}^1$. Its proof can be adapted directly from the proof of Proposition 4.6 of \cite{Carmona200471} (alternatively, Theorem~A.1 of~ \cite{BravoFernandezTineo07}) given for piecewise linear differential equations with two zones of linearity. 

\begin{proposition}\label{prop:der}
Assume $\mu>0$. The Poincaré map $P$ defined in \eqref{eq:Poincaremap}, the  displacement function given in \eqref{eq:disp_app}, and the Poincaré half-map $Q$ given in \eqref{eq:semipoinca} are continuously differentiable, and their respective derivatives can be written as 
\[
P'(x) = e^{2\pi\,a + (b-a)\, m(A_{in}(x))},
\]
\[
d'(x) = e^{2 \pi a + (b-a) m(A_{in}(x))}-1 
\]
and 
\[
Q'(x) = -e^{\pi\,a + (b-a)\, m(A_{in}(x)\cap [0,\pi])},
\]
where $A_{in}(x):=\{t\in[0,2\pi] \colon |u(t,0,x)|\leq 1 \}$ and $m(\cdot)$ denotes the measure of a measurable set in $\mathbb{R}$.
\end{proposition}

\begin{proof}
For the differentiability, we refer the reader to Proposition 4.6 of \cite{Carmona200471} or Theorem~A.1 of~ \cite{BravoFernandezTineo07}. Then,
we have that $u(t,0,x)$ is continuously differentiable with respect to $x$ and that
\begin{equation*}
u_x(t,0,x)= \exp\left(\int_0^t f'(u(s,0,x))\,ds \right),
\end{equation*}
where the subscript means partial derivative. Now, it suffices to observe that $P(x)=u(2\pi,0,x)$, that $Q(x)=-u(\pi,0,x)$, and that, for almost every $s$, $f'(u(s,0,x))=a$ if $|u(s,0,x)|>1$, and $f'(u(s,0,x))=b$ if $|u(s,0,x)|<1$.
\end{proof}

Fix $x\in\mathbb{R} $. Observe that if $P( x)= x$ and $P'( x)\ne 1$, then the periodic solution $u(t,0, x)$ is a hyperbolic limit cycle and this limit cycle is attractive (resp. repulsive) if $P'( x)<1$ (resp. $P'( x)>1$). In addition, the periodic solution is singular or non-hyperbolic (that is, $P'( x)=1$)  if and only if $2\pi a+(b-a)m(A_{in}(x))=0$.

The following result was originally published in~\cite[Teorema 1]{CFPT2}. As this volume is difficult to obtain, we include the proof for the sake of completeness.

\begin{proposition}[\cite{CFPT2}]\label{prop:sym}
Equation \eqref{eq:main1} has exactly one symmetric periodic solution, which we denote by $u^s(t)$.
\end{proposition}

\begin{proof} 
Firstly, we prove that for $x>0$ large enough, the solution $u(t,0,x)$ of~\eqref{eq:main1} is given by~\eqref{f3} for every $t\in[0,2\pi]$. 
Simply note that for $x>0$ large enough,
\[
e^{a t} x+ \int_{0}^t (b-a+\mu \sin( s))e^{a(t- s)}\,ds>1,\quad\text{for every }t\in[0,2\pi].
\]
Therefore, 
\[
Q(x)-x = -(e^{a \pi}+1) x - \int_{0}^\pi (b-a+\mu \sin( s))e^{a(\pi- s)}\,ds.
\]
Taking $x>0$ large enough, $Q(x)-x<0$.

Analogously, for $x<0$ sufficiently large in absolute value, we obtain $Q(x)-x>0$.

By the Intermediate Value Theorem, $Q$ has at least one fixed point $\bar x^s$. Now, from Proposition \ref{prop:der} it follows that $Q'(x)<0$ for all $x\in \mathbb{R}$ and so we conclude the uniqueness of the symmetric periodic orbit $u^s(t):=u\left(t,0,\bar x^s\right)$.
\end{proof}

As a direct consequence of Proposition \ref{prop:sym}, we have that equation \eqref{eq:main1} possesses at least one periodic solution, the symmetric one.

\begin{remark}\label{rema:signo}
The change of variable $t \to -t$ transforms the equation $x' = f(x) + \mu \sin(t)$ into $x' = f(-x) + \mu \sin(t)=-f(x) + \mu \sin(t)$, so we can choose the sign of $a$ or $b$. Although we are only interested in the case $\mu>0$, note that for $\mu=0$, $x'=f(x)$ is an autonomous scalar equation whose periodic solutions are constant. Moreover, the change of variable $t\to t+\pi$ transforms $x'=f(x)+\mu \sin(t)$ into $x'=f(x)-\mu \sin(t)$, so we can also choose the sign of $\mu$.
\end{remark}

\section{Uniqueness of Limit Cycles for the Case \texorpdfstring{$ab\geq0$}{ab >= 0}}\label{Sect3}

Note that \eqref{eq:main1} has a center if the displacement function \eqref{eq:disp_app} is zero in an entire interval. If the displacement function of \eqref{eq:main1} is identically zero, the center is said to be global.
\begin{theorem} \label{teor:unicilc}
Assume that $ab\geq 0$ and $\mu>0$. Then the dynamics of equation~\eqref{eq:main1} is one of the following:
\begin{enumerate}
\item Equation \eqref{eq:main1} has a global center if and only if $a=b=0$.
\item If $b = 0$ and $\mu < 1$, equation \eqref{eq:main1} has a center and no limit cycles.
\item In any other case, equation \eqref{eq:main1} has exactly one limit cycle, which is the symmetric periodic solution $u^s(t)$ given in Proposition \ref{prop:sym}. Moreover, the limit cycle is hyperbolic except when $b=0$ and $\mu=1$. It is globally asymptotically stable when $a+b<0$
and unstable when $a+b>0$.
\end{enumerate}
\end{theorem}
\begin{proof}  Note that by Remark~\ref{rema:signo}, we can assume, without loss of generality, that $a\leq0$ and $b\leq0$.

\medskip 
(1) Using the same reasoning as at the beginning of the proof of Proposition \ref{prop:sym}, one can see that, for sufficiently large negative values of $x<-1$, i.e., $x\ll -1$, $u(t,0,x)<-1$ for all $t\in[0,2\pi]$. Therefore, from \eqref{f1} we obtain
\[u(2\pi,0,x)= e^{2\pi a}x+\int_0^{2\pi} (a-b+\mu \sin( s)) e^{a(2\pi- s)}\,ds;\]
for all $x \ll -1$.  If equation \eqref{eq:main1} has a global center, then $u(t,0,x)$ is periodic for all $x\in\mathbb{R}$, that is, $x=u(2\pi,0,x)$   and so $a=b=0$. Conversely,  when $a=b=0$, equation \eqref{eq:main1} is $x'= \mu \sin(t)$, which clearly presents a global center.

\medskip 
(2) Assume $b=0$,  $a\neq0$, and $0\leq \mu<1$. Since $b=0$, $|u(t,0,x)|=|x+\mu(1-\cos(t))|\leq 1$, for every $x \in I:=[-1,1-2\mu]\subset [-1,1]$ and $t \in [0,2\pi]$. This means that $d(x)=0$ for $x\in I$ and so equation \eqref{eq:main1} has a center. If $x\notin I$, then $m(A_{in}(x)) < 2\pi$ and, by Proposition~\ref{prop:der}, $ad'(x)> 0$ and so equation \eqref{eq:main1} has no limit cycle.

\medskip 
(3) The remaining cases are:
\[ \text{(3.1)}\ b=0,\ a<0,\ \mu\geq 1,\qquad \text{(3.2)}\ b<0,\ a=0,\qquad \text{and}\qquad \text{(3.3)}\  b<0,\ a<0.\]
We treat them separately.

(3.1) Assume $b=0$, $a<0$, and $\mu>1$. If $m(A_{in}(x))=2\pi$ for some $x\in\mathbb{R}$, then, by continuity, $|u(t,0,x)|\leq1$ for every $t\in[0,2\pi]$. Consequently, by \eqref{f2}, $u(t,0,x)=x+\mu(1-\cos t)$, and hence $u(\pi,0,x)-u(0,0,x)=2\mu>2$. This is impossible because both values belong to $[-1,1]$. Therefore, $m(A_{in}(x))<2\pi$ and $d'(x)=e^{a(2\pi-m(A_{in}(x)))}-1<0$ for every $x\in\mathbb{R}$. 

If $\mu=1$, then $u(t,0,-1)=-\cos t$, $d(-1)=0$, and $d'(-1)=0$. For $x\neq-1$, the solution cannot remain in the central zone throughout $[0,2\pi]$, since this would require $-1\leq x$ and $x+2\leq1$, hence $x=-1$. Therefore, $m(A_{in}(x))<2\pi$ and $d'(x)<0$ for every $x\in\mathbb{R}\setminus\{-1\}$. In consequence, \eqref{eq:main1} has at most one limit cycle and, by Proposition \ref{prop:sym}, exactly one, the symmetric limit cycle. 

The limit cycle is globally asymptotically stable and is hyperbolic except when $\mu=1$.
 
(3.2) Recall that, by Proposition~\ref{prop:der}, $d'(x) = e^{b\, m(A_{in}(x))}-1$. Since $m(A_{in}(x)) \geq 0$, we have $d'(x)\leq 0$,  being equal to 0 only for those $x$ such that $|u(t,0,x)|\geq 1$, for every $t \in [0,2\pi]$. Therefore, by \eqref{f1} or \eqref{f3} as appropriate, we conclude that $d(x) =2\pi b\operatorname{sgn}(x)$ for those $x$ such that $|u(t,0,x)|\geq 1$ for $t \in [0,2\pi]$. Thus, equation \eqref{eq:main1} has at most one limit cycle $u(t,0,\bar{x})$ with $d'(\bar x)< 0$. From Proposition \ref{prop:sym}, the limit cycle is unique, hyperbolic and globally asymptotically stable. 

(3.3) By Proposition~\ref{prop:der}, 
\[d'(x) = e^{2 \pi a + (b-a) m(A_{in}(x))}-1=e^{a (2 \pi- m(A_{in}(x)) )  + b\, m(A_{in}(x))}-1.\] 
The derivative $d'(x)$ does not vanish because $0 \leq m(A_{in}(x))\leq2\pi.$ So, we conclude that $d'(x)$ has a constant sign. Therefore, equation \eqref{eq:main1} has at most one limit cycle, and its hyperbolicity and stability follow from $\operatorname{sgn}(d'(x))=\operatorname{sgn}(a)=\operatorname{sgn}(b)$. By Proposition \ref{prop:sym}, the conclusion follows.
\end{proof}

\section{Finiteness of the Number of Limit Cycles}\label{Sect:Finiteness}

Let us show that equation \eqref{eq:main1} has finitely many limit cycles. Indeed, we will prove that there exists a constant $L^\ast\in\mathbb{N}$, which does not depend on $a,b$ and $\mu$, such that equation \eqref{eq:main1} has no more than $L^\ast$ limit cycles.

Obviously, the one-region case of \eqref{eq:main1} has at most one limit cycle, since it corresponds to a linear differential equation. The general linear piecewise equation with two zones, $x'=f(t,x)$, where $f$ is linear and piecewise with respect to $x$ and $2\pi$-periodic with respect to $t$, is known to have any number of limit cycles, as shown in \cite{BravoFernandezTineo07,CollGasullProhens}. However, there are certain families for which finiteness of limit cycles occurs (see \cite{GasullZhao,BravoFernandezOjeda}). In particular, it is known that the two-region equations of the form $x' = g(x) + \mu \sin t$, where $g$ is a continuous piecewise linear function with two zones of linearity, have at most two limit cycles (see \cite{Carmona200471}).

Therefore, we only need to prove that equation \eqref{eq:main1} has finitely many limit cycles passing through all three zones. To avoid dealing with singular periodic solutions, we consider equation \eqref{eq:main1} perturbed by $\lambda$, that is, the differential equation
\begin{equation}\label{eq:main2}
x' = f(x) + \mu \sin(t) + \lambda,
\end{equation}
where $f(x)$ is defined as in \eqref{eq:f}. 

\begin{lemma}\label{lemalmitecycles}
If equation~\eqref{eq:main2} has at most $L$ non-singular periodic solutions for all sufficiently small nonzero $\lambda$, then equation~\eqref{eq:main1} has at most $L$ limit cycles.
\end{lemma}

\begin{proof}
Let $d(x,\lambda)$ denote the displacement function, with $a,b,\mu$ fixed. Since $d$ is continuously differentiable and $d_\lambda(x,\lambda)>0$, the equation $d(x,\lambda)=0$ locally defines a continuously differentiable function $\lambda(x)$. Consider any finite collection of $N$ limit cycles and choose disjoint neighborhoods of their initial values. Choose $\delta>0$ sufficiently small so that $\pm\delta$ are regular values of $\lambda(x)$ in these neighborhoods. At each selected zero, if $\lambda(x)$ changes sign, each perturbation produces at least one non-singular solution; otherwise, one of the two perturbations produces at least two. Thus, the perturbations $\lambda=\pm\delta$ together produce at least $2N$ non-singular periodic solutions. By hypothesis, each perturbation has at most $L$, so $N\leq L$. Since this holds for every finite collection of limit cycles, the conclusion follows.
\end{proof}

The proof of our finiteness theorem relies primarily on two technical results: Khovanski\u{\i}'s Theorem and a preparatory lemma. Both are introduced below.

\begin{theorem}[Khovanski\u{\i}]\label{Th Kho} 
Let $P_1 = \ldots = P_n = 0$ be a system of $n$ equations with $n$ real unknowns $x = (x_1, \ldots, x_n)$, where $P_i$ is a polynomial of degree $m_i$ in $n+k+2\rho$ real variables $x, y_1, \ldots, y_k,$ $u_1, \ldots, u_\rho, v_1, \ldots, v_\rho$, where $y_j = \operatorname{exp}\langle a_j, x \rangle,\ j = 1, \ldots, k$ and $u_q = \sin \langle b_q, x \rangle,$ $v_q = \cos \langle b_q, x \rangle,\ q = 1, \ldots, \rho$, where $\langle \cdot, \cdot \rangle$ denotes the usual scalar product. Then the number of isolated real solutions of this system in the region bounded by the inequalities $\vert \langle b_q, x \rangle \vert < \pi/2,\ q = 1, \ldots, \rho$ is finite and less than
\[
m_1 \cdots m_n \left(\sum_{i=1}^n m_i + \rho + 1\right)^{\rho+k} 2^{\rho + (\rho + k)(\rho + k -1)/2},
\]
\end{theorem}

\begin{proof}
For a proof, see \cite[pp. 14-16]{Khovanskii}.
\end{proof}

\begin{lemma}\label{lema:3zonas}
Assume $ab<0$. Consider $\lambda \in \mathbb{R}$ and denote by $u(t,\tau,x,\lambda)$ the maximal solution of the differential equation \eqref{eq:main2}
determined by the initial condition $u(\tau,\tau,x,\lambda)=x$.
If $u(t,0,x,\lambda)$ is a periodic solution of equation \eqref{eq:main2}
that crosses both $x=1$ and $x=-1$, then there exist
$t_1<t_2<t_3<t_4<t_1+2\pi$, $t_1\in[0,2\pi)$, such that
\begin{equation}\label{eq:3z}
\begin{split}
e^{-b t_1} (g(t_1,b)+\lambda) + e^{-b t_2} (g(t_2+\pi,b)-\lambda) &= 0; \\
e^{-a t_2} (g(t_2+\pi,a)-\lambda) - e^{-a t_3} (g(t_3+\pi,a)-\lambda) &= 0; \\
e^{-b t_3} (g(t_3+\pi,b)-\lambda) + e^{-b t_4} (g(t_4,b)+\lambda) &= 0; \\
e^{-a t_1} (g(t_1,a)+\lambda) - e^{-a (t_4-2\pi)} (g(t_4,a)+\lambda) &= 0,
\end{split}
\end{equation}
where
\begin{equation}\label{ecu:fung}
g(t,s) := \frac{\mu s\big(s\sin(t)+\cos(t)\big)+b(s^2+1)}{s^2+1}.
\end{equation}
\end{lemma}

\begin{proof}
If $u(t,0,x,\lambda)$ is a periodic solution of equation \eqref{eq:main2}
that crosses both $x=1$ and $x=-1$, then there exist
$t_1<t_2<t_3<t_4<t_1+2\pi$, $t_1\in[0,2\pi)$, such that
$|u(t,t_1,1,\lambda)|<1$ for $t\in(t_1,t_2)\cup(t_3,t_4)$,
$u(t,t_1,1,\lambda)<-1$ for $t\in(t_2,t_3)$,
$u(t,t_1,1,\lambda)>1$ for $t\in(t_4,t_1+2\pi)$, and
\begin{equation}\label{sis:3z}
\left\{
\begin{aligned}
u(t_2,t_1,1,\lambda) &= -1, \\
u(t_3,t_2,-1,\lambda) &= -1, \\
u(t_4,t_3,-1,\lambda) &= 1, \\
u(t_1+2\pi,t_4,1,\lambda) &= 1.
\end{aligned}
\right.
\end{equation}
Figure~\ref{fig:three-zone-orbit} illustrates these crossings and the
successive passages through the three regions.

\begin{figure}[htbp]
\centering
\begin{tikzpicture}[x=1.2cm,y=2cm,line cap=round,line join=round,scale=.65]
  \draw[->,>=stealth,blue!45!black,line width=1.3pt] (-1.8,0) -- (8.2,0);
  \draw[->,>=stealth,blue!45!black,line width=1.3pt] (0,-1.85) -- (0,2.35);
  \draw[blue!45!black,thin] (-1.5,1) -- (7.8,1);
  \draw[blue!45!black,thin] (-1.5,-1) -- (7.8,-1);
  \draw[blue!45!black,dashed] (6.4,0) -- (6.4,1);
  \foreach \time in {1.6,2.8}
    \draw[blue!45!black,dashed] (\time,0) -- (\time,-1);
  \draw[blue!45!black,dashed] (3.95,0) -- (3.95,1);
  \draw[red!65!black,line width=1.1pt]
    (0,1)
    .. controls (0.533333,0.245528) and (1.066667,-0.508944) .. (1.6,-1)
    .. controls (1.816667,-1.199492) and (2.033333,-1.355509) .. (2.25,-1.35)
    .. controls (2.433333,-1.345338) and (2.616667,-1.225028) .. (2.8,-1)
    .. controls (3.183333,-0.529488) and (3.566667,0.398835) .. (3.95,1)
    .. controls (4.333333,1.601165) and (4.716667,1.875172) .. (5.1,1.85)
    .. controls (5.533333,1.821544) and (5.966667,1.410772) .. (6.4,1);
  \foreach \point in {(0,1),(1.6,-1),(2.8,-1),(3.95,1),(6.4,1)}
    \fill[blue!45!black] \point circle[radius=2pt];
  \node[blue!45!black,above left] at (0,1) {$1$};
  \node[blue!45!black,below left] at (0,-1) {$-1$};
  \node[blue!45!black,below right] at (0,0) {$t_1$};
  \node[blue!45!black,below right] at (1.6,0) {$t_2$};
  \node[blue!45!black,below right] at (2.8,0) {$t_3$};
  \node[blue!45!black,below right] at (3.95,0) {$t_4$};
  \node[blue!45!black,below right] at (6.4,0) {$t_1+2\pi$};
\end{tikzpicture}
\caption{Schematic graph of a periodic solution crossing the boundaries $x=\pm1$.
The periodic solution lies in $|x|<1$ on $(t_1,t_2)\cup(t_3,t_4)$, in $x<-1$
on $(t_2,t_3)$, and in $x>1$ on $(t_4,t_1+2\pi)$.
The marked points correspond to the matching conditions in
\eqref{sis:3z}; dashed lines indicate the crossing times.}
\label{fig:three-zone-orbit}
\end{figure}

Let us see how the first equation of \eqref{eq:3z} is derived from the
first equation of \eqref{sis:3z}.

Since $|u(t,t_1,1,\lambda)|<1$ for $t\in(t_1,t_2)$, we find that
$u(t,t_1,1,\lambda)$ is the solution of the linear equation
\[
x' = b\,x + \mu\sin(t) + \lambda
\]
that passes through $(t_1,1)$, that is,
\[
u(t,t_1,1,\lambda) = \frac{e^{-b t_1}}{b}
\left(e^{b t}(g(t_1,b)+\lambda)
- e^{b t_1}(g(t,b)+\lambda-b)\right),
\]
for every $t\in(t_1,t_2)$, where the function $g$ is given in
\eqref{ecu:fung}. Thus, the condition $u(t_2,t_1,1,\lambda)=-1$ becomes
\[
\frac{e^{-b t_1}}{b}
\left(e^{b t_2}(g(t_1,b)+\lambda)
- e^{b t_1}(g(t_2,b)+\lambda-b)\right) = -1;
\]
equivalently,
\[
e^{b t_2}(g(t_1,b)+\lambda)
- e^{b t_1}(g(t_2,b)+\lambda-2b) = 0.
\]
Therefore, multiplying by $e^{-b(t_1+t_2)}$ and noticing that
$-g(t_2,b)+2b = g(t_2+\pi,b)$, we get
\[
e^{-b t_1}(g(t_1,b)+\lambda)
+ e^{-b t_2}(g(t_2+\pi,b)-\lambda) = 0,
\]
that is, the first equation of \eqref{eq:3z}.

Finally, arguing similarly in the remaining regions, where equation
\eqref{eq:main2} is linear and an explicit expression for
$u(t,t_1,1,\lambda)$ can be obtained from \eqref{eq:sollinear}, we can
conclude that system \eqref{sis:3z} is equivalent to system \eqref{eq:3z}.
\end{proof}

We are now in a position to state and prove the finiteness theorem.

\begin{theorem}\label{Th:finiteness}
There exists a constant $L^\ast\in\mathbb{N}$, which does not depend on $a,b$ and $\mu$, such that equation \eqref{eq:main1} has no more than $L^\ast$ limit cycles.
\end{theorem}

\begin{proof}
Note, firstly, that using Theorem \ref{teor:unicilc}, equation \eqref{eq:main1} has at most one limit cycle when $ab\geq 0$. Therefore, by Lemma~\ref{lemalmitecycles}, it suffices to obtain a uniform upper bound, independent of $\lambda,a,b$ and $\mu$, for the number of non-singular periodic solutions of \eqref{eq:main2} when $ab<0$. Note that for $\mu=0$, equation \eqref{eq:main2} does not have non-constant periodic solutions. 

Taking into account that a linear differential equation in the form \eqref{eq:linearode} has at most one non-singular periodic solution, equation \eqref{eq:main2} has at most three one-zonal non-singular periodic solutions (that is, each of them lives in a zone of linearity). From Theorem 4.19 of \cite{Carmona200471}, equation \eqref{eq:main2} has at most four two-zonal non-singular periodic solutions (each of them only visits two zones of linearity). Therefore, to prove the finiteness of the number of non-singular periodic solutions for equation \eqref{eq:main2} for the case $ab<0$, we must prove the finiteness of the number of non-singular periodic solutions of equation \eqref{eq:main2} that cross both $x=1$ and $x=-1$.

From Lemma \ref{lema:3zonas}, each periodic solution of \eqref{eq:main2} that crosses both $x=1$ and $x=-1$ corresponds to a solution of system \eqref{eq:3z} which is a system of four equations in four real unknowns $t_1, \ldots, t_4$, where the equations are polynomials of degree $2$ in the following sixteen real variables $y_{i1}=e^{-a t_i}, y_{i2}=e^{-b t_i}, u_i=\sin(t_i)$ and $v_i=\cos(t_i),\ i\in\{1,2,3,4\}$. So, Theorem~\ref{Th Kho} applies after a fixed finite covering of $[0,4\pi)^4$ and suitable translations of the variables. Thus, the number of non-singular (in particular, isolated) real solutions of \eqref{eq:3z} in this domain is bounded independently of $\lambda,a,b$ and $\mu$, and the conclusion follows.
\end{proof}

\section{Number of Limit Cycles for Extreme Values of \texorpdfstring{$\mu$}{mu}}\label{Sect:Compacity}
 
The cases $ab \geq 0$ have been completely studied in Section~\ref{Sect3}.
So, assuming that $ab<0$, in this section, we demonstrate that exactly three limit cycles exist for $\mu$ close to zero (Theorem \ref{Th:-mu_small}), and exactly one non-singular limit cycle exists for $\mu$ sufficiently large (Theorem \ref{Th:-mu_large}).

Let us start with a useful lemma.

\begin{lemma}\label{le:1}
Consider equation \eqref{eq:linearode} with $p\ne 0$. The only periodic solution of equation \eqref{eq:linearode} is 
\begin{equation}
\label{eq:opeqlinear}
    v(t)=-\frac{q}{p}-\frac{\mu(\cos(t)+p\sin(t))}{p^2+1}.
\end{equation}
Moreover, this periodic solution is non-singular (that is, hyperbolic), globally asymptotically stable when $p<0$, and unstable when $p>0$.
\end{lemma}

\begin{proof}
From expression \eqref{eq:sollinear}, one can deduce that the Poincaré map for equation \eqref{eq:linearode} is given by the affine map 
\[
\tilde P(x):=\frac{\left(e^{2 \pi  p}-1\right)
   \left(p^2+1\right) q+e^{2 \pi  p}
   p \left(\mu +p^2
   x+x\right)-\mu 
   p}{p^3+p}
\]
which, for $p\ne 0$, has a unique fixed point $\bar x=-q/p-\mu\left(p^2+1 \right)^{-1}$. This fixed point provides the unique periodic solution  $v(t)$ given in \eqref{eq:opeqlinear}. Since the derivative of $\tilde P$ is $\tilde P'(x)=e^{2\pi p}$, the conclusion of the hyperbolicity and stability of the periodic solution $v$ is direct.
\end{proof}

Due to the previous lemma, we can determine when equation~\eqref{eq:main1} has a periodic solution in the regions $x\geq 1$ and $x\leq -1$.

\begin{proposition}\label{prop:1}
Consider equation \eqref{eq:main1} with $ab<0$ and $\mu > 0$. The following statements are equivalent.
\begin{enumerate}
\item There exists a periodic solution $u^+$ of equation \eqref{eq:main1} such that $u^+(t) \geq 1$ for all $t\in\mathbb{R}$.
\item There exists a periodic solution $u^-$ of equation \eqref{eq:main1} such that $u^-(t) \leq -1$ for all $t\in\mathbb{R}$.
\item $\mu\leq -b \sqrt{a^2+1}/a$.
\end{enumerate}
In these cases, the periodic solutions are the only ones in the corresponding zone of linearity;  specifically,
\begin{equation}
\label{eq:opumas}
    u^+(t) = \frac{a-b}{a} - \frac{\mu(\cos(t)+a\sin(t))}{a^2+1}\quad \text{and}\quad u^-(t) = -u^+(t+\pi).
\end{equation}
In addition, these periodic solutions are hyperbolic, asymptotically stable when $b>0$, and unstable when $b<0$.
\end{proposition}

\begin{proof}
Since the equivalence of items (1) and (2) follows by symmetry, it suffices to prove the equivalence of items (1) and (3). 

If equation \eqref{eq:main1} possesses a periodic solution $u^+$ such that $u^+(t) \geq  1$ for all $t\in\mathbb{R}$, then the solution $u^+$ is a periodic solution of the linear differential equation $x'= ax+(b-a)+\mu \sin(t)$. Thus, by Lemma \ref{le:1}, the expression for the periodic solution $u^+$ is provided in \eqref{eq:opumas} by substituting $p=a$ and $q=b-a$ in the periodic solution $v$ of \eqref{eq:opeqlinear}. 

A straightforward computation indicates that the global minimum of the function $u^+$ given in \eqref{eq:opumas} is 
\[H:=\frac{a-b}{a}-\frac{\mu}{\sqrt{1+a^2}}.\]
Therefore, the function $u^+$ is a periodic solution of \eqref{eq:main1} if and only if $H\geq 1$, which is equivalent to the inequality of item (3). The expression for the periodic solution $u^-$ follows by symmetry, and, since $a b < 0$, the conclusion for the hyperbolicity and stability of the periodic solutions is a direct consequence of Lemma \ref{le:1}.
\end{proof}

Next, we study when limit cycles of equation  \eqref{eq:main1} exist in the zone $|x|\leq 1$. The proof of the next result is similar to the proof of Proposition \ref{prop:1} and is omitted.

\begin{proposition}\label{prop:2}
Consider equation \eqref{eq:main1} with $b\ne0$ and $\mu > 0$. There exists an isolated periodic solution $u$ of equation \eqref{eq:main1} such that $|u(t)| \leq 1$ for all $t\in\mathbb{R}$ if and only if $\mu\leq \sqrt{b^2+1}$. In these cases, the periodic solution is the only one in the  zone $|x|\leq 1$; specifically,
\[
u(t) = -\frac{\mu (\cos(t)+b\sin(t))}{b^2+1},
\]
which is the symmetric periodic solution $u^s$ given in Proposition \ref{prop:sym}. In addition, this periodic solution is hyperbolic, asymptotically stable when $b<0$ and unstable when $b>0$.
\end{proposition}

Summarizing, we can state our result for the existence of three limit cycles for equation \eqref{eq:main1}.

\begin{theorem}\label{Th:-mu_small}
Suppose that $ab < 0$ and $\mu > 0$. Equation~\eqref{eq:main1} has exactly three limit cycles, each of them located in a zone of linearity, if and only if
\begin{equation}\label{eq:mu}
\mu \leq \min\left\{ \frac{-b\sqrt{a^2 + 1}}{a}, \sqrt{b^2 + 1} \right\}.
\end{equation}
These three limit cycles are hyperbolic; those in the zones $x \geq 1$ and $x \leq -1$ are asymptotically stable when $b > 0$ and unstable when $b < 0$, and the limit cycle contained in the region $|x| \leq 1$ is asymptotically stable when $b < 0$ and unstable when $b > 0$.
\end{theorem}

\begin{proof}
By Propositions~\ref{prop:1} and \ref{prop:2}, equation~\eqref{eq:mu} is satisfied if and only if \eqref{eq:main1} admits three limit cycles each one contained in a single linearity zone. In that situation, no additional limit cycles can occur. In fact, \eqref{eq:main1} cannot have further limit cycles lying entirely within one zone, by linearity; nor can it have any limit cycle crossing exactly two zones, by convexity (see, e.g., \cite{Carmona200471}). Moreover, because the symmetric limit cycle is located in the central zone, there are no limit cycles intersecting all three zones.
\end{proof}

To conclude this section, we prove the uniqueness of limit cycles for $\mu$ sufficiently large, where the size depends on parameters $a$ and $b$. 

\begin{theorem}\label{Th:-mu_large} 
Suppose that $ab<0$ and $\mu > 0$. There exists a value $\bar{\mu}=\bar{\mu}(a,b)> 0$ such that if $\mu >\bar\mu$, then equation \eqref{eq:main1} has only one limit cycle $u(t,0,x)$. Moreover, this limit cycle is symmetric, hyperbolic, globally asymptotically stable when $b>0$ and unstable when $b<0$. Moreover,  
the initial condition $x$ satisfies $|x|\leq \mu(1+1/|a|)$.
\end{theorem}

\begin{proof}
By Remark~\ref{rema:signo}, we may assume $a>0$, $b<0$. Note that the change of variables reverses the stability.

The change of variable $\bar x= x/\mu$ transforms equation \eqref{eq:main1} into the equation
\begin{equation}\label{eq:cvmain1}
\bar x'=g(\bar x,\mu)+\sin t 
\end{equation}
where 
\[
g(\bar x,\mu)=
\begin{cases}
a\, \bar x+\frac{a-b}{\mu},\quad \text{if }\bar x\leq -\frac{1}{\mu},\\
b\, \bar x,\quad \text{if }-\frac{1}{\mu}<\bar x < \frac{1}{\mu},\\
a\, \bar x+\frac{b-a}{\mu},\quad \text{if } \bar x\geq \frac{1}{\mu}.
\end{cases}	
\]
Rename the new variable $\bar x$ as $x$. Then equation \eqref{eq:cvmain1} becomes $x'=g(x,\mu)+\sin t$.

Note that the limit differential equation of \eqref{eq:cvmain1} as $\mu\to\infty$  is the scalar linear differential equation 
\begin{equation}\label{eq:linear_limit}
 x'=a x+\sin t.
\end{equation}
By means of Lemma \ref{le:1}, 
equation \eqref{eq:linear_limit} has a unique limit cycle $u_l$ and this limit cycle is hyperbolic, unstable,
 symmetric and crosses the straight line $x=0$ transversely. 

As the vector field \eqref{eq:cvmain1} is continuous with respect to $\nu = 1/\mu$, the solutions of \eqref{eq:cvmain1} in a neighborhood $\mathcal U$ of $u_l$ for $t\in[0,2\pi]$ are also transversal to $x=0$ for $t\in[0,2\pi]$. 
Moreover, for every solution~$u$ of equation ~\eqref{eq:cvmain1} contained in $\mathcal U$ for $t\in[0,2\pi]$,
\[\lim_{\nu\to 0}m(A_{in}(\nu,u))=0,\] 
where $A_{in}(\nu,u)=\{t\in[0,2\pi]\colon |u(t)|\leq \nu\}$. As the change of variables preserves the parametrization of time, Proposition~\ref{prop:der} and the hyperbolicity of $u_l$ imply the existence of $\tilde \mu=\mu(a,b)>0$ such that equation \eqref{eq:cvmain1} has a unique limit cycle in $\mathcal U$ for $\mu>\tilde \mu$. 

In addition, this limit cycle is symmetric, hyperbolic, and crosses transversely the separation straight lines $x=-1/\mu$ and $x=1/\mu$. Since $b<0$, it is unstable.

The vector field defined by equation \eqref{eq:linear_limit} in the lines $ x=\pm(1+1/a)$ satisfies
\[
\begin{array}{l}
a x'|_{1+1/a}=a(a+1+\sin t)>0, \quad \mbox{for all} \quad t\in\mathbb{R},
\\
\noalign{\medskip}
a x'|_{-(1+1/a)}=-a(a+1-\sin t)<0, \quad \mbox{for all} \quad t\in\mathbb{R},
\end{array}
\]
which means that the band $| x|\leq 1+1/a$ is a  negatively invariant region for equation \eqref{eq:linear_limit}, because $a>0$. Thus, there exists a value $\hat \mu=\hat\mu(a,b)>0$ such that the periodic solutions of equation \eqref{eq:cvmain1} are bounded in the band $| x|\leq 1+1/a$ for $\mu>\hat \mu$. 

Now, if there exists a sequence $\mu_n\to+\infty$, with $\mu_n>\max\{\tilde\mu,\hat\mu\}$ for every $n$, such that equation \eqref{eq:cvmain1} with $\mu=\mu_n$ has a periodic solution $u_n$ other than the symmetric periodic solution (see Proposition \ref{prop:sym}), then, by compactness, there exists a convergent subsequence $\{u_{n_k}\}$ and, using continuity with respect to the parameters and initial conditions, this subsequence converges to $u_l$, the unique periodic solution of equation \eqref{eq:linear_limit}. Therefore, for $\mu=\mu_{n_k}$ sufficiently large, equation \eqref{eq:cvmain1} has two periodic solutions in a neighbourhood of $u_l$, and this is impossible based on the hyperbolicity of the solution $u_l$. 

Therefore, there exists a value $\bar{\mu}=\bar{\mu}(a,b)>\max\{\tilde \mu, \hat \mu\}$ such that if $\mu >\bar\mu$, then equation \eqref{eq:main1} has only one limit cycle $u(t,0,x)$. The hyperbolicity and stability of this limit cycle are straightforward, and the bound of the old variable $|x|$ follows from the invariance of the band $|x|/\mu=|\bar x| \leq 1+1/a$.
\end{proof}

Observe that as $\mu\to+\infty$, the unique periodic solution of the rescaled equation~\eqref{eq:cvmain1} converges uniformly on $[0,2\pi]$ to the unique periodic solution of~\eqref{eq:linear_limit}.

\subsection*{Proof of Theorem~\ref{theo:tridimensional}.}

We will prove that periodic orbits of system \eqref{ecu:sisttriobsnocontb} living in the cylinder $y^2+z^2=\mu^2$ correspond to periodic solutions of the $2\pi$-periodic scalar piecewise linear differential equation
\begin{equation}\label{ecu:main3}
	x'=ax+(b-a)\mathop{\rm sat}(x)+\mu\sin t,
\end{equation}
being $a,b\in\mathbb{R}$, $\mu>0$, and $\mathop{\rm sat}(\cdot)$ the normalized saturation function given in \eqref{ecu:sat}.
		
The solution of the two-dimensional system 
\[
\left\{
\begin{array}{l}
	y'=z, \\
	\noalign{\smallskip}
	z'=-y,
\end{array}
\right.
\]
with initial condition $(y(0),z(0))=(0,-\mu)$ is  given by $(y(t),z(t))=(-\mu\sin t,-\mu\cos t)$ and so periodic orbits of system \eqref{ecu:sisttriobsnocontb} living in the cylinder $y^2+z^2=\mu^2$ correspond to periodic solutions of the $2\pi$-periodic scalar piecewise linear differential equation~\eqref{ecu:main3}. 

To conclude, it suffices to apply Theorem~\ref{teor:unicilc} to obtain (1), and Theorems~\ref{Th:-mu_small} and \ref{Th:-mu_large}, together with Remark~\ref{rema:signo}, to obtain (2).

\section{First Order Perturbation}\label{Sect:Melnikov}

Define $f(x)$ as in \eqref{eq:f} and consider,
for $|\varepsilon|$ sufficiently small,
the differential equation~\eqref{eq:ef},
\[
x'=\varepsilon f(x)+\mu\sin t.
\]
In this section, we compute the first-order Melnikov
function or averaging function for this deformation
and prove the following result.

\begin{theorem}\label{theo:Melnikov}
Assume $ab<0$, and define
\begin{equation}\label{ecu:defmu1yc}
c:=\frac{\pi b}{b-a}, \qquad
\mu_1:=\frac{c}{\sin(c)}, \qquad
\mu_2:=\frac{1}{\cos(c/2)}.
\end{equation}
Then $0<\mu_2<\mu_1$, and for each
$\mu>0$ with $\mu\neq\mu_1,\mu_2$, there exists
$\varepsilon_0(\mu)>0$ such that the following
statements hold whenever
$0<|\varepsilon|<\varepsilon_0(\mu)$:
\begin{enumerate}
    \item if $0<\mu<\mu_2$, then equation
    \eqref{eq:ef} has exactly three limit cycles.
    \item if $\mu_2<\mu<\mu_1$, then equation
    \eqref{eq:ef} has exactly five limit cycles.
    \item if $\mu>\mu_1$, then equation
    \eqref{eq:ef} has exactly one limit cycle.
\end{enumerate}
Moreover, the symmetric limit cycle is asymptotically
stable if $\varepsilon b(\mu_2-\mu)<0$ and unstable
if $\varepsilon b(\mu_2-\mu)>0$.
When ordered by their initial values at $t=0$,
the limit cycles have alternating stability.
\end{theorem}

\subsection{Melnikov function: general properties}

Let us write $u(t,x,\mu,\varepsilon)$ for the solution of equation \eqref{eq:ef} such that $u(0,x,\mu,\varepsilon)=x$. 
For $\varepsilon = 0$,  equation \eqref{eq:ef} has a  global center, since \[u(t,x,\mu,0) = \mu(1-\cos t) + x,\] for every $x \in \mathbb{R}$. 

The change of variable $x\to x-\mu\cos t $ transforms the differential equation \eqref{eq:ef} into 
\begin{equation}\label{eq:ef1}
x'=\varepsilon f(x-\mu\cos t),
\end{equation}
and, therefore, the study of the periodic solutions of equation \eqref{eq:ef} for $|\varepsilon|$ sufficiently small can be carried out using the averaging theory (or Melnikov theory). From now on, $u(t,x,\mu,\varepsilon)$ denotes the solution of equation~\eqref{eq:ef1} with $u(0,x,\mu,\varepsilon)=x$. Notice that the derivative with respect to $\varepsilon$ of the solution $u(t,x,\mu,\varepsilon)$ for $t=2\pi$ and $\varepsilon=0$ coincides with the averaging function $M$, i.e. 
\begin{equation}\label{ecu:Melnikovfun}
u_{\varepsilon}(2\pi,x,\mu,0)=M(x,\mu):=\int_0^{2\pi} f(x-\mu\cos t)\,dt
\end{equation}
For this reason, function $M$ is also called the first-order Melnikov function (see, for instance,  \cite{Carmona201344,Novaes20241486,Blows1994341}).

Since
$$
u(2\pi,x,\mu,\varepsilon)-x=\varepsilon(M(x,\mu)+\varepsilon R(x,\mu,\varepsilon)),
$$
applying the Implicit Function Theorem to the normalized
equation $M(x,\mu)+\varepsilon R(x,\mu,\varepsilon)=0$,
one can deduce that the simple zeros with respect to $x$ of $M$ provide limit cycles of equation \eqref{eq:ef} for $0<|\varepsilon|$ sufficiently small. See e.g. \cite{Neto1980}.

The main objective of this Section is the proof of Theorem \ref{theo:Melnikov}, which relies primarily on the analysis of the Melnikov function $M$, and more specifically, on the precise description of its zero set. The symmetries of $M$ allow us to restrict the description
of $M=0$ to the first quadrant of the $(x,\mu)$-plane. Indeed,

\begin{proposition}\label{prop:sim}
$ M(-x,\mu) = -  M(x,\mu)$ for every $(x,\mu)\in\mathbb{R}^2$. In particular, $ M(0,\mu) =0$ for all $\mu>0$. 
\end{proposition}

\begin{proof}
Since the function $f$ given in \eqref{eq:f} is odd, we have that 
\begin{align*}
 M(-x,\mu) = \int_0^{2\pi} f(-x-\mu\cos t)\,dt = -\int_0^{2\pi} f(x+\mu \cos t)\,dt= \\ 
 -\int_0^{2\pi} f(x-\mu\cos(t-\pi))\,dt = -\int_{-\pi}^{\pi} f(x-\mu\cos t)\,dt   = - M(x,\mu).
\end{align*}
Hence, $ M(0,\mu) =0$ for all $\mu >0$.
\end{proof}

Now, we show that the function $M$ is continuously differentiable and provide expressions for the derivatives of $ M$ with respect to $x$ and with respect to $\mu$.

\begin{proposition}
\label{prop:dertildeM}
Function $M$ defined in \eqref{ecu:Melnikovfun} is continuously differentiable in $\mathbb{R}\times(0,+\infty)$. Moreover, the derivatives of $M$ with respect to $x$ and with respect to $\mu$ are given, respectively, by
\begin{equation}
\label{ecu:DerMtildex}
M_x(x,\mu)=2\pi a-(a-b)m\left(A_{in}(x,\mu)\right)
\end{equation}
and
\begin{equation}
\label{ecu:DerMtildemu}
M_\mu(x,\mu)=(a-b)\int_{A_{in}(x,\mu)}\cos t\,dt,
\end{equation}
where
\begin{equation}
\label{ecu:Atilde}
A_{in}(x,\mu)=\left\{t\in[0,2\pi]:|x-\mu\cos t|\leq1\right\}
\end{equation}
and $m(\cdot)$ denotes, as usual, the measure of a measurable set in $\mathbb{R}$.
\end{proposition}

\begin{proof}
The set of times
satisfying $x-\mu\cos t=\pm1$ has measure zero.
Since $f$ is Lipschitz continuous, differentiation
under the integral sign and dominated convergence
show that $M$ is continuously differentiable.
The derivative of $M$ with respect to $x$ is
\[
\begin{split}
M_x(x,\mu)&=\int_0^{2\pi}f'(x-\mu\cos t)\,dt
=\int_{A_{in}(x,\mu)}b\,dt
+\int_{[0,2\pi]\setminus A_{in}(x,\mu)}a\,dt\\
&=b\,m\left(A_{in}(x,\mu)\right)
+a\,(2\pi-m\left(A_{in}(x,\mu)\right))
=2\pi a-(a-b)m\left(A_{in}(x,\mu)\right),
\end{split}
\]
where the set $A_{in}(x,\mu)$ is defined in \eqref{ecu:Atilde}.

Since the integral of $\cos t$ in $[0,2\pi]$ is zero, the derivative of $M$ with respect to $\mu$ is
\[
\begin{split}
M_\mu(x,\mu)&=\int_0^{2\pi}-f'(x-\mu\cos t)\cos t\,dt
=\int_{A_{in}(x,\mu)}-b\cos t\,dt
+\int_{[0,2\pi]\setminus A_{in}(x,\mu)}-a\cos t\,dt\\
&=(a-b)\int_{A_{in}(x,\mu)}\cos t\,dt.
\end{split}
\]
\end{proof}

\begin{remark}
The expressions for the partial derivatives of $ M$ given in Proposition \ref{prop:dertildeM} allow us to obtain a new expression for $ M$ depending on $ M_x$ and $ M_\mu$. Indeed, by using the expression for function $f$ given in \eqref{eq:f}, it follows that 
\[
\begin{array}{ccl}
 M(x,\mu)&=&{\displaystyle \int_0^{2\pi} f(x-\mu\cos t)\,dt=\int_{A^+(x,\mu)} (a(x-\mu\cos t)+(b-a)) dt}+ \\
\noalign{\medskip}
& & {\displaystyle \int_{ A_{in}(x,\mu)} b(x-\mu\cos t) dt}+ {\displaystyle \int_{A^-(x,\mu)} (a(x-\mu\cos t)+(a-b)) dt}, \\
\end{array}
\]
where $ A_{in}(x,\mu)$ is given in \eqref{ecu:Atilde}, 
\begin{equation}
\label{eq: A+-}
\begin{array}{lcl}
A^+(x,\mu)& =& \{ t\in[0,2\pi]: x-\mu\cos t \geq 1\} \qquad \mbox{and}  \\
\noalign{\medskip}
A^-(x,\mu)& = & \{ t\in[0,2\pi]: x-\mu\cos t \leq -1\}.
\end{array}
\end{equation}
Hence, since $A^+(x,\mu)\cup A^-(x,\mu) =[0,2\pi]\setminus A_{in}(x,\mu)$ up to a set of measure zero, one gets
\[
\begin{array}{ccl}
 M(x,\mu)&=&{\displaystyle \int_{[0,2\pi]\setminus A_{in}(x,\mu)} ax \, dt -\mu\left(\int_{[0,2\pi]\setminus A_{in}(x,\mu)} a\cos t \,dt+ \int_{ A_{in}(x,\mu)} b\cos t \,dt \right)} +\\
\noalign{\medskip}
& & {\displaystyle \int_{ A_{in}(x,\mu)} bx \, dt+ \int_{A^+(x,\mu)} (b-a) \, dt+ \int_{A^-(x,\mu)} (a-b)\,  dt }\\
\end{array}
\]
and, from expressions \eqref{ecu:DerMtildex}, \eqref{ecu:DerMtildemu}, one has 
\begin{equation}
\label{ecu:Mtildeversion}
 M(x,\mu)=x M_x(x,\mu)+\mu M_\mu(x,\mu)-(a-b)\left(m\left(A^+(x,\mu) \right)- m\left(A^-(x,\mu) \right)\right).
\end{equation}
\end{remark}
 
By Proposition \ref{prop:sim}, from now on, we assume $x\geq0$. Moreover, by  Remark~\ref{rema:signo}, we also assume $a>0$, $b<0$, whenever it simplifies the exposition. 

\begin{remark}\label{rema:derMmu}
The following observations are made with regard to the derivative with respect to $\mu$ of the function $M$. 

Let us consider the periodic function $v(t)=x-\mu\cos t$ and the set $A_{in}(x,\mu)$ defined in \eqref{ecu:Atilde}.
Function $v$ is one--zonal if and only if
$m(A_{in}(x,\mu))=0$ or $m(A_{in}(x,\mu))=2\pi$.
Thus, from expression for $M_\mu$ given in \eqref{ecu:DerMtildemu}, one has $M_\mu(x,\mu)=0$ when $v$ is one--zonal.

If $v$ is two--zonal, then $\emptyset\ne A_{in}(x,\mu)\varsubsetneq [0,2\pi]$ and the equation $|v(t)|=1$ has exactly two solutions $t_1=t_1(x,\mu)$ and $t_2=t_2(x,\mu)=2\pi-t_1(x,\mu)$ in the interval $(0,2\pi)$, with $t_1\in(0,\pi)$.
Thus, it is immediate to observe that
$[0,2\pi]\setminus A_{in}(x,\mu)=(t_1,2\pi-t_1)$.

When $[0,2\pi]\setminus A_{in}(x,\mu)=(t_1,2\pi-t_1)$, from \eqref{ecu:DerMtildemu}, it is deduced
\begin{equation}
\label{eq:dertildeMmu2zb}
M_\mu(x,\mu)=(a-b)\int_{A_{in}(x,\mu)}\cos t\,dt
=-(a-b)\int_{t_1}^{2\pi-t_1}\cos t\,dt
=2(a-b)\sin t_1>0.
\end{equation}

If $v$ is three--zonal, then $\emptyset\ne A_{in}(x,\mu)\varsubsetneq [0,2\pi]$ and the equation $|v(t)|=1$ has exactly four solutions in the interval $(0,2\pi)$. Namely, $\hat t_1=t_1(x,\mu),\hat t_2=t_2(x,\mu)$, $\hat t_3=2\pi-\hat t_2$, and $\hat t_4=2\pi-t_1$ (with $0<\hat t_1<\hat t_2<\pi$), where
\begin{equation}
\label{eq:t1t23z}
|x-\mu\cos\hat t_1|=|x-\mu\cos\hat t_2|=1
\quad \mbox{and}\quad
(x-\mu\cos\hat t_1)(x-\mu\cos\hat t_2)=-1.
\end{equation}
Hence, $A_{in}(x,\mu)=[\hat t_1,\hat t_2]\cup[2\pi-\hat t_2,2\pi-\hat t_1]$ and
\begin{equation}
\label{eq:dertildeMmu3z}
M_\mu(x,\mu)=(a-b)\int_{A_{in}(x,\mu)}\cos t\,dt
=2(a-b)\left(\sin\hat t_2-\sin\hat t_1\right).
\end{equation}
\end{remark}

\subsection{Characterization of the Zero Set of the Function \texorpdfstring{$M$}{M}}
Some basic properties of the function $M$ are collected in the next result. In particular, we prove that its zero set corresponding to non-zero values of $x$ is contained within a compact subset of the plane.

\begin{proposition}
\label{prop:tildeM}
Suppose that  $a>0$, $b<0$, $x\geq0$, and $\mu > 0$. Then, 
function $M$ defined in \eqref{ecu:Melnikovfun} satisfies the following properties:
\begin{enumerate}
\item \label{item1proptildeM} If $\mu\leq 1$, then   $ M_x(0,\mu)=2\pi b < 0$.
\item \label{item3proptildeM} Suppose that  $v(t)=x-\mu\cos t $, with $x> 0$, is one--zonal. Then,  $ M(x,\mu) =0$ if and only if $x=1-b/a$ and $\mu\leq -b/a$. In addition, $ M_x(1-b/a,\mu)=2\pi a > 0$ for $\mu\leq -b/a$.
\item\label{itemdermuproptildeM}
$M_\mu(0,\mu)=0$.
Moreover, if the set $A_{in}(x,\mu)$ defined in
\eqref{ecu:Atilde} has measure $0$ or $2\pi$,
then $M_\mu(x,\mu)=0$.
In addition, $M_\mu(x,\mu)>0$ if $x>0$ and
$0<m(A_{in}(x,\mu))<2\pi$.
\item\label{itemderxproptildeM} $M_x(0,\mu)=2\pi b$ for all $\mu\leq 1$ and $M_x(0,\mu)$ is strictly increasing for $\mu> 1$.
\item \label{item5proptildeM} If $x> 1-b/a$, then $M(x,\mu) > 0$ for all $\mu> 0$.
\item \label{item6proptildeM}There exists a value $\check \mu=\check \mu(a,b)>0$ such that $M(x,\mu)> 0$ for $x>0$ and $\mu>\check \mu$.
\end{enumerate}
\end{proposition}

\begin{proof}
\begin{enumerate}
\item
If $\mu\leq1$, then the function $v(t)=-\mu\cos t$ lives in the zone $|x|\leq1$, so $A_{in}(0,\mu)=[0,2\pi]$ and, from \eqref{ecu:DerMtildex}, $M_x(0,\mu)=2\pi b<0$.

\item
When the periodic function $v(t)=x-\mu\cos t$ is one--zonal, one gets
\[
M(x,\mu)=\int_0^{2\pi}f(x)\,dt=2\pi f(x).
\]
Thus, for $x>0$, $M(x,\mu)=0$ if and only if $x=1-b/a$, because $ab<0$. Now, it is direct to observe that $v(t)=(1-b/a)-\mu\cos t$ is one--zonal if and only if $\mu\leq-b/a$. Therefore, for $\mu\leq-b/a$, it has
$m(A_{in}(1-b/a,\mu))=0$ and so, from expression \eqref{ecu:DerMtildex}, $M_x(1-b/a,\mu)=2\pi a>0$ and the proof of item \eqref{item3proptildeM} is concluded.

\item
Assume $\mu>0$. By Proposition~\ref{prop:sim}, $M(0,\mu)=0$ for all $\mu>0$ and so $M_\mu(0,\mu)=0$ for all $\mu>0$. By means of Remark \ref{rema:derMmu}, it follows that $M_\mu(x,\mu)=0$ if
$m(A_{in}(x,\mu))=0$ or $m(A_{in}(x,\mu))=2\pi$.
Moreover, when $0<m(A_{in}(x,\mu))<2\pi$ and $v(t)=x-\mu\cos t$ is two--zonal, from \eqref{eq:dertildeMmu2zb}, we obtain $M_\mu(x,\mu)>0$. Finally, if
$0<m(A_{in}(x,\mu))<2\pi$, the periodic function $v$ is three--zonal and $M_\mu(x,\mu)=0$, from \eqref{eq:dertildeMmu3z}, we deduce the relationship $\hat t_2=\pi-\hat t_1$, which implies, using \eqref{eq:t1t23z}, $x=0$.
Therefore, for $x>0$, continuity and positivity in the two--zonal case imply positivity in the three--zonal case, and the proof of item \eqref{itemdermuproptildeM} is completed.

\item
From~\eqref{ecu:DerMtildex} and $m(A_{in}(0,\mu))=2\pi$ for $\mu\leq1$, we have $M_x(0,\mu)=2\pi b$. Since $m(A_{in}(0,\mu))$ is strictly decreasing for $\mu>1$, we conclude.

\item
For $\mu=0$, we have
\[
M(x,0)=\int_0^{2\pi}f(x)\,dt=2\pi f(x)
\]
and so, if $x>1-b/a$, then it follows directly that $M(x,0)>0$. For fixed $x>1-b/a$, as $\mu$ increases, $v(t)=x-\mu\cos t$ changes from one-zonal, to two-zonal and finally three-zonal (see figure~\ref{fig:zeroset}).
By item \eqref{itemdermuproptildeM}, $M_\mu(x,\mu)>0$ whenever $v$ is two-zonal or three-zonal. Therefore, $M(x,\mu)>0$.
\color{black}

\item
In view of item \eqref{item5proptildeM}, we only need to prove the result for
$x\in (0,1-b/a]$.

\begin{figure}[htbp]
\centering
\begin{tikzpicture}[scale=1.0]

\pgfmathsetmacro{\tone}{acos(-0.35)} 
\pgfmathsetmacro{\ttwo}{acos(0.65)}  

\draw[->] (0,-2.2) -- (6.7,-2.2) node[right] {$t$};
\draw[->] (0,-2.2) -- (0,2.5)
  node[above] {$v(t)$};

\foreach \x/\lab in
{
0/{$0$},
1.5708/{$\frac{\pi}{2}$},
3.1416/{$\pi$},
6.2832/{$2\pi$}
}
{
  \draw (\x,-2.28)--(\x,-2.12);
  \node[below] at (\x,-2.28) {\lab};
}

\foreach \y/\lab in
{
-1.7/{$x-\mu$},
-1/{$-1$},
0.3/{$x$},
1/{$1$}
}
{
  \draw (-0.08,\y)--(0.08,\y);
  \node[left] at (-0.08,\y) {\lab};
}

\draw[dashed,gray] (0,1) -- (6.3,1);
\draw[dashed,gray] (0,-1) -- (6.3,-1);

\draw[densely dotted] (0,-1.7) -- (6.3,-1.7);
\draw[densely dotted] (0,0.3) -- (pi/2,0.3);

\draw[densely dotted] (pi/2,-2.1) -- (pi/2,0.3);

\draw[black,thick,samples=300,domain=0:360,smooth]
plot ({\x*pi/180},{0.3-2*cos(\x)});

\fill (\ttwo*pi/180,-1) circle (1.5pt);
\fill (\tone*pi/180,1) circle (1.5pt);

\draw[dashed] (\ttwo*pi/180,-1)--(\ttwo*pi/180,-2.2);
\node[below] at (\ttwo*pi/180,-2.2) {$t_1$};

\draw[dashed] (\tone*pi/180,1)--(\tone*pi/180,-2.2);
\node[below] at (\tone*pi/180,-2.2) {$t_2$};

\end{tikzpicture}
\caption{A three-zonal $x-\mu\cos\,t$ with $x\in(0,1-b/a)$.}
\label{fig:three-zonal}
\end{figure}

Assume that $\mu$ is large enough so that $x-\mu\cos t$ is three-zonal (see Figure~\ref{fig:three-zonal}).

Denote $t_1(x,\mu),t_2(x,\mu)$ the solutions of $x-\mu\cos t=-1,x-\mu\cos t=1$, respectively, where $t\in(0,\pi)$. That is,
\[
t_1(x,\mu)=\arccos\left(\frac{x+1}{\mu}\right),
\qquad
t_2(x,\mu)=\arccos\left(\frac{x-1}{\mu}\right).
\]
Then,
\[
m(A_{in}(x,\mu))=2(t_2(x,\mu)-t_1(x,\mu)),
\quad
m(A^+(x,\mu))=2(\pi-t_2(x,\mu)),
\quad
m(A^-(x,\mu))=2t_1(x,\mu).
\]

For $x\in[0,1-b/a]$, we have
\[
\max\left\{
\left|\frac{x+1}{\mu}\right|,
\left|\frac{x-1}{\mu}\right|
\right\}
\leq\frac{2-b/a}{\mu}\longrightarrow0
\qquad\text{as }\mu\to+\infty.
\]
Since $\arccos$ is continuous at $0$, both
$t_1(x,\mu)$ and $t_2(x,\mu)$ converge uniformly
to $\pi/2$ for $x\in[0,1-b/a]$.
Therefore, $m(A_{in}(x,\mu))\to0$ uniformly on this interval
and, by \eqref{ecu:DerMtildex},
\[
M_x(x,\mu)\longrightarrow2\pi a>0
\]
uniformly on the same interval.
Thus, there exists $\check\mu=\check\mu(a,b)>0$ such that
$M_x(x,\mu)>0$ for every $x\in[0,1-b/a]$ and
$\mu>\check\mu$.
Since $M(0,\mu)=0$, it follows that
\[
M(x,\mu)=\int_0^x M_x(s,\mu)\,ds>0
\]
for every $x\in(0,1-b/a]$ and $\mu>\check\mu$.
Together with item \eqref{item5proptildeM}, this proves the result.
\end{enumerate}
\end{proof}

\begin{remark}
It should be noted that item  \eqref{item1proptildeM} of Proposition \ref{prop:tildeM} provides the limiting counterpart of Proposition \ref{prop:2}  tailored to the perturbed differential equation \eqref{eq:ef} because the function $v(t)=-\mu\cos t$ is symmetric and the inequality $\mu\leq \sqrt{b^2+1}$ becomes the inequality $\mu\leq 1$ when $b$ is transformed into $\varepsilon b$ and $\varepsilon$ goes to zero.  Analogously,  item  \eqref{item3proptildeM} of Proposition \ref{prop:tildeM} is the limiting counterpart of Proposition \ref{prop:1}. In this case, the inequality $\mu\leq -b \sqrt{a^2+1}/a$ becomes the inequality $\mu\leq -b/a$ when $a$ and $b$  are transformed into $\varepsilon a$ and $\varepsilon b$, respectively, and $\varepsilon$ tends to zero. Finally, item \eqref{item6proptildeM} of Proposition \ref{prop:tildeM} is consistent with Theorem \ref{Th:-mu_large}.
\end{remark}

Continuing with the aim of characterizing the zero set of the function $ M$, we shall prove that this set is defined by a unique function $\mu$ that depends on $x$, for $0<x<1-b/a$. Naturally, the symmetry property in Proposition~\ref{prop:sim} allows us to extend this to $x<0$.  The following Theorem refers to a complete characterization of the zero set of $M$ (see Figure~\ref{fig:zeroset} for a scheme of the zero set of $M$). 

\begin{theorem}\label{theo:zerosettildeM}
Suppose that  $a>0$,  $b<0$  are fixed, and that $x\geq0$, $\mu > 0$. Then, there exists a unique continuous function $\varphi:[0,1-b/a]\longrightarrow (0,+\infty)$ such that the zero set of function $ M$ defined in \eqref{ecu:Melnikovfun} can be written in the form:
\begin{equation}\label{eq:M1}
 M^{-1}(\{0\})=\left\{(0,\mu)\in\mathbb{R}^2:\mu > 0 \right\}\cup Z_1,
\end{equation}
where  
\[
Z_1=\left\{(1-b/a,\mu)\in\mathbb{R}^2: 0 < \mu\leq-b/a \right\}\cup \left\{(x,\mu)\in\mathbb{R}^2: \mu=\varphi(x), 0\leq x < 1-b/a \right\}.
\]

Moreover, function $\varphi$ satisfies the following properties:
\begin{enumerate}
\item \label{item1thM}
$\varphi$ is differentiable in $(0,1-b/a)$.
\item $\varphi$ possesses a unique critical point $x_1\in (0,1-b/a)$. Furthermore,  $\varphi'(x)>0$ for $x\in(0,x_1)$, $\varphi'(x)<0$ for $x\in(x_1,1-b/a)$ and  $x_1$ is given by 
\begin{equation}
\label{ecu:defx1}
x_1:=1-\mu_1\cos (c),
\end{equation}
where the values $\mu_1$ and $c$ are defined in \eqref{ecu:defmu1yc}.
\item  \label{item3thM} $\varphi$ is continuous in $[0,1-b/a]$, $\varphi(1-b/a)=-b/a$ and $\varphi(0)=\mu_2$, with $\mu_2$ defined in \eqref{ecu:defmu1yc}.
\end{enumerate}
\end{theorem}

\begin{figure}[h]
\begin{center}
\begin{tikzpicture}[
    x=2.0cm,y=1.00cm,
    >=Stealth,
    scale=.75,
    every node/.style={font=\large}
]

\def\s{0.74}       
\def\c{1.30}       
\def\Y{3.00}       
\def\xone{1.66}    

\def\mzero{1.7568} 
\def\mb{2.2973}    
\def\muTwo{2.76}
\def\muOne{3.17}

\def\xmax{4.55}
\def\ymax{3.65}

\path[use as bounding box]
    (-0.62,-3.78) rectangle (4.70,3.82);


\fill[gray!42]
    (0,\c)
    -- (\mzero,0)
    -- (0,-\c)
    -- cycle;

\fill[gray!42]
    (0,\c)
    -- ({(\ymax-\c)/\s},\ymax)
    -- (0,\ymax)
    -- cycle;

\fill[gray!42]
    (0,-\c)
    -- ({(\ymax-\c)/\s},-\ymax)
    -- (0,-\ymax)
    -- cycle;

\fill[green!35]
    (\mzero,0)
    -- (\xmax,{\s*(\xmax-\mzero)})
    -- (\xmax,{-\s*(\xmax-\mzero)})
    -- cycle;


\draw[black!45,line width=0.45pt]
    (0,\c)
    -- ({(\ymax-\c)/\s},\ymax);

\draw[black!45,line width=0.45pt]
    (0,\c)
    -- (\xmax,{\c-\s*\xmax});

\draw[black!45,line width=0.45pt]
    (0,-\c)
    -- (\xmax,{-\c+\s*\xmax});

\draw[black!45,line width=0.45pt]
    (0,-\c)
    -- ({(\ymax-\c)/\s},-\ymax);


\draw[dashed,line width=0.8pt]
    (\mb,\Y) -- (\mb,-\Y);

\draw[dashed,line width=0.8pt]
    (0,\xone) -- (\muOne,\xone);

\draw[dashed,line width=0.8pt]
    (\muOne,\xone) -- (\muOne,-2.10);


\draw[blue!70!black,line width=1.55pt]
    (0,\Y)
    -- (\mb,\Y)

    .. controls (2.86,3.00)
                and (3.17,2.55)
    .. (\muOne,\xone)

    .. controls (3.17,0.96)
                and (\muTwo,0.55)
    .. (\muTwo,0)

    .. controls (\muTwo,-0.55)
                and (3.17,-0.96)
    .. (\muOne,-1.66)

    .. controls (3.17,-2.55)
                and (2.86,-3.00)
    .. (\mb,-\Y)

    -- (0,-\Y);

\draw[blue!70!black,line width=1.55pt] (0,0) -- (4.28,0);
\draw[-{Stealth[length=3.2mm,width=2.2mm]}, line width=0.9pt] (0,-3.63) -- (0,3.56);

\draw[-{Stealth[length=3.2mm,width=2.2mm]},line width=0.9pt] (0,0) -- (4.52,0);
\draw[black!45,line width=0.4pt] (0,-\ymax) rectangle (\xmax,\ymax);
\fill (0,\Y)       circle[radius=3pt];
\fill (0,\xone)    circle[radius=3pt];
\fill (\mb,0)      circle[radius=3pt];
\fill (\muTwo,0)   circle[radius=3pt];
\fill (\muOne,0)   circle[radius=3pt];
\node[above right=-1pt] at (0,3.52) {$x$};
\node[below left=1pt] at (4.52,0) {$\mu$};
\node[left=5pt] at (0,\Y) {$1-\dfrac{b}{a}$};
\node[left=5pt] at (0,\xone) {$x_1$};
\node[below=4pt] at (\mb,0) {$-\dfrac{b}{a}$};
\node[below=4pt] at (\muTwo,0) {$\mu_2$};
\node[below right=3pt] at (\muOne,0) {$\mu_1$};
\node[anchor=west] at (3.20,2.48) {$\mu=\varphi(x)$};
\end{tikzpicture}
\end{center}
\caption{
Schematic representation of the zero set
of the function $M$ defined in \eqref{ecu:Melnikovfun} for $ab<0$
in the case $-b/a<\mu_2$.
The values $\mu_1$ and $\mu_2$ are given in \eqref{ecu:defmu1yc},
and the value $x_1$ and the existence of $\varphi$ are given in
Theorem \ref{theo:zerosettildeM}.
The shaded regions intersecting the axis $\mu=0$ correspond to
$(x,\mu)$ for which $v(t)=x-\mu\cos t$ lies in one zone,
and the outer shaded regions correspond to $v(t)$ lying in three zones.}\label{fig:zeroset}
\end{figure}

\subsection{Proof of Theorem \ref{theo:zerosettildeM}.}

Before presenting the proof of Theorem~\ref{theo:zerosettildeM},
we require a technical result concerning the periodic function $v_1(t)=x_1+\mu_1\cos t$, which lives exactly in the zones $x\geq1$ and $-1<x<1$.

\begin{lemma}
\label{lema:v1}
Suppose $a>0$ and $b<0$. Consider the values   $\mu_1$   and $c$ given in \eqref{ecu:defmu1yc} and the value  $x_1$ given in \eqref{ecu:defx1}. Then, the inequality $\mu_1\cos c>b/a$ holds and  $v_1(t)=x_1+\mu_1\cos t $ lives only in the zones $x\geq 1$ and $-1<x<1$.
\end{lemma}

\begin{proof}
The inequality $\mu_1\cos c> b/a$ can be written into the form 
\[
c\left(\frac{\cos c}{\sin c}+\frac{1}{\pi-c}\right)>0.
\]
Since $ab<0$, the value $c$ defined in \eqref{ecu:defmu1yc} belongs to the open interval $(0,\pi)$ and so the previous inequality is transformed in 
\[
(\pi-c)\cos c+\sin c>0.
\]
Taking into account that the function $g(s)=(\pi-s)\cos s+\sin s$ is strictly decreasing, $g(0)=\pi$, and $g(\pi)=0$, the inequality $\mu_1\cos c> b/a$ holds.

Since $c\in(0,\pi)$ and $\mu_1>0$, the periodic function $v_1(t)=x_1+\mu_1\cos t =1-\mu_1\cos (c)+\mu_1\cos t$ satisfies $v_1(c)=1$ and $v_1'(c)=-\mu_1\sin(c)<0$. Thus, the function $v_1$ lives in the zones $x\geq 1$ and $x<1$. Now, we prove that this function does not cross the straight line $x=-1$. To this end, we see that its global minimum $v_1(\pi)=1-\mu_1(1+\cos c)$ is strictly greater than $-1$. That is, we show that
\[
\frac{c (1+\cos c)}{\sin c}<2.
\]
The function $h(s)=s (1+\cos s)/\sin s$ is strictly decreasing, $\lim_{s\to 0^+}h(s)=2$, and $\lim_{s\to \pi^-}h(s)=0$, and so the result follows.
\end{proof}

\medskip
Now, we can proceed with the proof of Theorem \ref{theo:zerosettildeM}.
First, we prove the existence of the function $\varphi$ defined in the open interval $(0,1-b/a)$. We then extend the definition of the function $\varphi$ to the endpoints of this interval. 

Suppose that $x\in(0,1-b/a)$ is fixed. It is direct to see that $\operatorname{sgn}\left(M(x,0)\right) = \operatorname{sgn}(f(x))<0$. From item \eqref{item6proptildeM} of Proposition \ref{prop:tildeM}, we have $M(x,\mu)>0$ when $\mu$ is sufficiently large. By continuity of function $M$, there exists a value $\mu(x)>0$ such that $M(x,\mu(x))=0$. By using Remark \ref{rema:derMmu} and item \eqref{itemdermuproptildeM} of Proposition \ref{prop:tildeM}, it is easy to conclude that the value $\mu(x)$ is unique and so we get a unique function $\varphi:(0,1-b/a)\longrightarrow(0,+\infty)$ such that if $x\in(0,1-b/a)$, $\mu>0$ and $M(x,\mu)=0$, then $\mu=\varphi(x)$. Moreover, from items \eqref{item3proptildeM} and \eqref{item5proptildeM} of Proposition \ref{prop:tildeM}, if $x>0$ and $\mu > 0$, then $M(x,\mu)=0$ if and only if $(x,\mu)\in Z_1$. Finally, by Proposition~\ref{prop:sim}, we obtain~\eqref{eq:M1}.

\medskip 
Now, we prove items \eqref{item1thM}-\eqref{item3thM}. 

\begin{enumerate}
\item Notice that, by using item \eqref{item3proptildeM} of Proposition \ref{prop:tildeM}, the function $v(t)=x-\varphi(x)\cos t$, with $x\in(0,1-b/a)$, is not one--zonal, that is, $0<m(A_{in}(x,\varphi(x)))<2\pi$, and so, from item \eqref{itemdermuproptildeM} of  Proposition~\ref{prop:tildeM}, $ M_\mu(x,\varphi(x))\ne0$. This implies, from the Implicit Function Theorem, that function $\varphi$ is differentiable in the open interval $(0,1-b/a)$ and so item \eqref{item1thM} is true. 

\item In the following, we prove that the function $\varphi$
has a unique critical point in the open interval $(0,1-b/a)$.
For this purpose, we consider the periodic function
$v(t)=x-\varphi(x)\cos t$.
Note that function $v$ is not one--zonal, its global maximum is
$v(\pi)=x+\varphi(x)$ and its global minimum is
$v(0)=x-\varphi(x)$.

\medskip
If the function $v$ is two--zonal, then there exists
$t_1=t_1(x)\in(0,\pi)$ such that
\begin{equation}
\label{eq:t12z}
x-\varphi(x)\cos t_1=1
\qquad \mbox{with} \qquad
x-\varphi(x)=-1+z,
\end{equation}
where $z\in(0,2)$.
Thus, the set $A_{in}$ defined in \eqref{ecu:Atilde} is
$A_{in}(x,\varphi(x)) =[0,t_1]\cup[2\pi-t_1,2\pi]$, 
and the sets $A^+$ and $A^-$ given in \eqref{eq: A+-} are
$A^+(x,\varphi(x))=[t_1,2\pi-t_1]$ and $A^-(x,\varphi(x))=\emptyset$, respectively. Therefore, by using expressions \eqref{ecu:Mtildeversion}
and \eqref{ecu:DerMtildemu}, it is deduced that
\begin{equation}
\label{2ecu:2z}
0=M(x,\varphi(x))
=xM_x(x,\varphi(x))
+2\varphi(x)(a-b)\sin t_1
-2(a-b)(\pi-t_1).
\end{equation}

If the function $\varphi$ possesses a critical point
$x_1\in(0,1-b/a)$, with $v(t)=x_1-\varphi(x_1)\cos t$ two--zonal, then
$M_x(x_1,\varphi(x_1))=0$ and the relationship
\eqref{2ecu:2z} is
\begin{equation}
\label{ecu:ecut1yfi}
\varphi(x_1)\sin t_1=\pi-t_1.
\end{equation}

Taking into account that $A_{in}(x_1,\varphi(x_1))
=[0,t_1]\cup[2\pi-t_1,2\pi]$, using expression \eqref{ecu:DerMtildex},
$M_x(x_1,\varphi(x_1))=0$ if and only if
\[
t_1=\frac{\pi a}{a-b}=\pi-c,
\]
with $c$ given in \eqref{ecu:defmu1yc}.
Thus, from expression \eqref{ecu:ecut1yfi},
$\varphi(x_1)=\mu_1$, where $\mu_1$ is defined in
\eqref{ecu:defmu1yc}, and, from the first expression of
\eqref{eq:t12z} for $x=x_1$, it follows that
$x_1=1+\varphi(x_1)\cos t_1
=1-\mu_1\cos c$.
Notice that $x_1>0$, since $\sin c-c\cos c>0$ for
$c\in(0,\pi)$. Moreover, Lemma \ref{lema:v1} gives
$x_1<1-b/a$ and shows that the function
$x_1-\mu_1\cos t=v_1(t+\pi)$ is two--zonal.
Conversely, substituting these values into
\eqref{ecu:DerMtildex} and \eqref{2ecu:2z} gives
$M_x(x_1,\mu_1)=M(x_1,\mu_1)=0$.
By the uniqueness of $\varphi(x_1)$, it follows that
$\varphi(x_1)=\mu_1$ and $\varphi'(x_1)=0$.

The same identities hold by continuity when $v(0)=-1$.
A critical point in this case would also yield $x=x_1$
and $\varphi(x)=\mu_1$, whereas Lemma \ref{lema:v1}
gives $x_1-\mu_1>-1$.

\medskip
When function $v$ is three--zonal, there exist
$\hat t_1=\hat t_1(x),\hat t_2=\hat t_2(x)$
(with $0<\hat t_1<\hat t_2<\pi$), such that
$v(\hat t_1)=-1$ and $v(\hat t_2)=1$, that is,
\begin{equation}
\label{eq:t1t23zb}
x-\varphi(x)\cos\hat t_1=-1
\qquad \mbox{and} \qquad
x-\varphi(x)\cos\hat t_2=1.
\end{equation}
Hence, the set $A_{in}$ defined in \eqref{ecu:Atilde} is
$
A_{in}(x,\varphi(x))
=[\hat t_1,\hat t_2]
\cup[2\pi-\hat t_2,2\pi-\hat t_1],
$
and the sets $A^+$ and $A^-$ given in \eqref{eq: A+-} are
$
A^+(x,\varphi(x))=[\hat t_2,2\pi-\hat t_2]$, and
$
A^-(x,\varphi(x))
=[0,\hat t_1]\cup[2\pi-\hat t_1,2\pi]$, respectively. Therefore, by using expressions \eqref{ecu:Mtildeversion}
and \eqref{ecu:DerMtildemu}, it is deduced that
\begin{equation}
\label{3ecu:z}
0=M(x,\varphi(x))
=xM_x(x,\varphi(x))
+2\varphi(x)(a-b)(\sin\hat t_2-\sin\hat t_1)
+2(a-b)(\hat t_1+\hat t_2-\pi).
\end{equation}

If the function $\varphi$ possesses a critical point
$x_1\in(0,1-b/a)$, with
$v(t)=x_1-\varphi(x_1)\cos t$ three--zonal, then
$M_x(x_1,\varphi(x_1))=0$ and the relationship
\eqref{3ecu:z} is
\begin{equation}
\label{ecu:ecut1yt2}
-\varphi(x_1)(\sin\hat t_2-\sin\hat t_1)
=\hat t_1+\hat t_2-\pi.
\end{equation}

By means of the expressions of \eqref{eq:t1t23zb},
it follows that
\begin{equation}
\label{ecu:ecut1yt2b}
\varphi(x_1)(\cos\hat t_2-\cos\hat t_1)=-2.
\end{equation}

From expressions \eqref{ecu:ecut1yt2}
and \eqref{ecu:ecut1yt2b}, one has
\[
-4\varphi(x_1)
\sin\left(\frac{\hat t_2-\hat t_1}{2}\right)
\left[
\cos\left(\frac{\hat t_2+\hat t_1}{2}\right)
+\left(\frac{\hat t_2+\hat t_1}{2}-\frac{\pi}{2}\right)
\sin\left(\frac{\hat t_2+\hat t_1}{2}\right)
\right]=0.
\]
Since $\varphi(x_1)>0$, $0<\hat t_1<\hat t_2<\pi$,
and the unique solution $s\in(0,\pi)$ of equation
$\cos s+(s-\pi/2)\sin s=0$ is $s=\pi/2$,
it follows that $\hat t_1+\hat t_2=\pi$.
From expressions \eqref{eq:t1t23zb} for $x=x_1$,
it is deduced that $x_1=0$.
Because this is impossible, the function $\varphi$
has no critical points corresponding to three--zonal
functions $v$.

This concludes that $x_1=1-\mu_1\cos c$ is the unique
critical point of the function $\varphi$.
To determine the sign of $\varphi'$, note that in the
two--zonal case,
\[
M_x(x,\varphi(x))=2\pi a-2(a-b)t_1(x),
\qquad
M_\mu(x,\varphi(x))=2(a-b)\sin t_1(x)>0.
\]
Differentiating $x-\varphi(x)\cos t_1(x)=1$ at $x=x_1$
gives
\[
t_1'(x_1)=-\frac{1}{\mu_1\sin c}.
\]
Consequently,
\[
\left.\frac{d}{dx}M_x(x,\varphi(x))\right|_{x=x_1}
=\frac{2(a-b)}{\mu_1\sin c}>0.
\]
Since
\[
\varphi'(x)
=-\frac{M_x(x,\varphi(x))}{M_\mu(x,\varphi(x))}
\]
and $M_\mu(x,\varphi(x))>0$, it follows that
$\varphi'>0$ immediately to the left of $x_1$
and $\varphi'<0$ immediately to its right.
By continuity and the uniqueness of the critical point,
these signs hold throughout $(0,x_1)$
and $(x_1,1-b/a)$, respectively.

\item From the monotonicity of $\varphi$ in the intervals $(0,x_1)$ and $(x_1,1-b/a)$, it follows directly that the limits
\[
\lim_{x\to 1-\frac{b}{a}^-}\varphi(x)
\qquad \mbox{and} \qquad
\lim_{x\to 0^+}\varphi(x)
\]
exist. Moreover, using
items \eqref{item3proptildeM} and
\eqref{itemdermuproptildeM} of Proposition \ref{prop:tildeM},
one has $\lim_{x\to1-\frac{b}{a}^-}\varphi(x)=-b/a$ and
$L=\lim_{x\to0^+}\varphi(x)$.
Therefore, by defining $\varphi(1-b/a)=-b/a$ and $\varphi(0)=L$,
it follows that $\varphi$ is continuous in $[0,1-b/a]$.

In addition, since
$M(x,\varphi(x))=M(0,\varphi(x))=0$,
Rolle's Theorem provides $\xi_x\in(0,x)$ such that
$M_x(\xi_x,\varphi(x))=0$.
Taking the limit as $x\to0^+$ and using the continuity of $M_x$,
one obtains
\[
M_x(0,\varphi(0))=0.
\]
By item \eqref{item1proptildeM} of Proposition \ref{prop:tildeM},
this implies $\varphi(0)>1$, and hence
$v(t)=0-\varphi(0)\cos t$ is three--zonal.
By direct computation, it follows that
$\hat t_1+\hat t_2=\pi$. From expression \eqref{ecu:DerMtildex},
we have $\hat t_1=c/2$, and, from
\eqref{eq:t1t23zb}, one deduces $\varphi(0)=\mu_2$, with $\mu_2$
given in \eqref{ecu:defmu1yc}.
\end{enumerate}

\subsection{Proof of Theorems \ref{theo:Melnikov} and \ref{theo:tridimensional2}}

To prove Theorem~\ref{theo:tridimensional2} we use the 1-dimensional reduction considered in the proof of Theorem~\ref{theo:tridimensional}. Then, we apply Theorem~\ref{theo:Melnikov} to equation~\eqref{eq:ef}.

Now, we prove Theorem~\ref{theo:Melnikov}. From Proposition~\ref{prop:sim} and Remark~\ref{rema:signo}, we assume $a>0$, $b<0$, $x\geq0$.

By means of the definition of $\mu_1$ and $\mu_2$ given in \eqref{ecu:defmu1yc} one obtains $0<\mu_2<\mu_1$. From the characterization of the zero set of the function $M$ established in Proposition~\ref{prop:tildeM}, items \eqref{item3proptildeM} and \eqref{itemderxproptildeM}, and Theorem \ref{theo:zerosettildeM}, it follows directly that this set exhibits the structure depicted in Figure \ref{fig:zeroset}. Moreover, $M_x(x,\mu)\neq0$ for $(x,\mu)$ in the zero set of the function $M$ for $\mu\neq\mu_1,\mu_2$:
\begin{itemize}
\item $M_x(1-b/a,\mu)=2\pi a$ for $\mu\leq-b/a$.
\item Differentiating $M(x,\varphi(x))=0$, we obtain
\[
M_x(x,\varphi(x))=-M_\mu(x,\varphi(x))\varphi'(x).
\]
Since $M_\mu(x,\varphi(x))\neq0$, Theorem~\ref{theo:zerosettildeM} implies that $M_x(x,\varphi(x))\neq0$ for $\varphi(x)\neq\mu_1,\mu_2$.
\item By item \eqref{itemderxproptildeM} of Proposition~\ref{prop:tildeM}, $M_x(0,\mu)$ has a unique zero for $\mu>0$. By Theorem~\ref{theo:zerosettildeM} and continuity, it is easy to obtain that $M_x(0,\mu)=0$ if and only if $\mu=\mu_2$. Moreover,
$\operatorname{sign}(M_x(0,\mu))=\operatorname{sign}(b(\mu_2-\mu))$.
\end{itemize}

Hence, one can deduce that:
\begin{enumerate}
    \item if $0<\mu<\mu_2$, then function $ M(x,\mu)$ has exactly three simple zeros,
    \item if $\mu_2< \mu <\mu_1$, then function $ M(x,\mu)$ has exactly five simple zeros,
    \item if $\mu>\mu_1$, then function $ M(x,\mu)$ has exactly one simple zero.
\end{enumerate}

Let $\mu=\mu_0$. 
We  prove that the initial conditions of limit cycles of \eqref{eq:ef1} are contained in a compact interval independent of $\varepsilon$.  If $u(t)$ is a limit cycle of \eqref{eq:ef1}, and $t_M,t_m$ are points where $u$ attains its maximum and minimum, then $u'(t_M)=u'(t_m)=0$. Hence $u(t_M)-\mu_0\cos t_M$ and $u(t_m)-\mu_0\cos t_m$ are zeros of $f$. Therefore the maximum and the minimum of $u$ remain in a compact interval independent of $\varepsilon$, and so does the whole solution. Let $K\subset\mathbb R$ be such a compact interval containing all initial conditions of limit cycles of \eqref{eq:ef1}. Enlarging $K$ if necessary, we may also assume that it
contains all the zeros of $M(\cdot,\mu_0)$.

Finally, we prove that if $M(x,\mu_0)$ has $k$ simple zeroes then \eqref{eq:ef1} for $\mu=\mu_0$ has exactly $k$ limit cycles, for sufficiently small nonzero $\varepsilon$.

We have
\begin{align*}
d(x,\mu,\varepsilon)
&:=u(2\pi,x,\mu,\varepsilon)-x\\
&=\varepsilon\left(u_{\varepsilon}(2\pi,x,\mu,0)
+\varepsilon R(x,\mu,\varepsilon)\right)\\
&=\varepsilon\left(M(x,\mu)+\varepsilon R(x,\mu,\varepsilon)\right),
\end{align*}
where $\varepsilon R(x,\mu,\varepsilon)\to 0$ as $\varepsilon\to0$, uniformly for $x\in K$ with $\mu=\mu_0$ fixed.

Define
\[
F(x,\mu,\varepsilon)=
\begin{cases}
d(x,\mu,\varepsilon)/\varepsilon,&\varepsilon\neq0,\\
M(x,\mu),&\varepsilon=0.
\end{cases}
\]
Moreover, Proposition~\ref{prop:der} and continuous dependence of the
solutions give
$F_x(x,\mu_0,\varepsilon)\to M_x(x,\mu_0)$
as $\varepsilon\to0$.

Assume that $M(x,\mu_0)$ has $k$ simple zeroes, $x_1<x_2<\cdots<x_k$.
By the Implicit Function Theorem applied to $F(x,\mu_0,\varepsilon)=0$,
there exist $\varepsilon_i>0$ and $\eta_i>0$ 
and a unique $x_i\colon(-\varepsilon_i,\varepsilon_i)\to\mathbb{R}$
with $x_i(0)=x_i$ such that, for $|x-x_i|<\eta_i$ and
$0<|\varepsilon|<\varepsilon_i$,
\[
d(x,\mu_0,\varepsilon)=0
\quad\text{if and only if}\quad
x=x_i(\varepsilon).
\]
We have that $[x_1,x_k]\subset K$.
By continuity of $F$, since
$u_{\varepsilon}(2\pi,x,\mu_0,0)\neq0$ for every $x\in K$,
$x\neq x_i$, we have $d(x,\mu_0,\varepsilon)\neq0$ whenever
$(x,\varepsilon)$, with $\varepsilon\neq0$, belongs to a
neighborhood of $(x,0)$.
By the compactness of $K$, we can find $\varepsilon_0>0$ such that
for $x\in K$ and $0<|\varepsilon|<\varepsilon_0$,
\[
d(x,\mu_0,\varepsilon)=0
\quad\text{if and only if}\quad
x=x_i(\varepsilon),\quad\text{for some }i=1,\dots,k.
\]
The conclusion follows taking into account that
all limit cycles have their initial conditions
in $K$ for $\varepsilon\neq0$.

To prove the stability assertions, it suffices to observe that,
for sufficiently small nonzero $\varepsilon$, the sign of
$d'(x_i(\varepsilon),\mu_0,\varepsilon)$ coincides with that of
$\varepsilon M_x(x_i,\mu_0)$, for all $i$.

The preceding argument shows that each simple zero of $M$
gives rise to a unique nearby hyperbolic limit cycle for
sufficiently small nonzero $\varepsilon$, and that no
additional limit cycles occur.
The proof of Theorem~\ref{theo:Melnikov} concludes
by applying the above result for $k=1,3,5$.

\section*{Acknowledgments}

J.L.B., M.F., and I.O. are partially supported by the project PID2023-151974NB-I00 funded by MICIU/
AEI/10.13039/ 501100011033/FEDER, UE. This work has been partially funded by the Junta de Extremadura through projects GR24042 (J.L.B. and M.F.) and GR24068 (I.O.), partially funded by the European Regional Development Fund (ERDF) ``A way to make Europe''.
V.C. is partially supported by the Ministerio de Ciencia, Innovaci\'on y Universidades, Plan Nacional I+D+I cofinanced with FEDER funds, in the frame of the project PID2024-158381NB-I00, and the Consejer\'{i}a de Educaci\'{o}n y Ciencia de la Junta de Andaluc\'{i}a (TIC-0130). V.C. thanks IMUS-Mar\'{i}a de Maeztu grant CEX2024-001517-M - Apoyo a Unidades de Excelencia Mar\'{i}a de Maeztu for supporting this research, funded by MICIU/AEI/ 10.13039/501100011033.

We thank the referees for their helpful comments and suggestions.

\section*{ORCID}
\noindent Jos\'e Luis Bravo --- \url{https://orcid.org/0000-0002-0012-7699}. \par
\noindent Victoriano Carmona --- \url{https://orcid.org/0000-0002-2889-913X}. \par
\noindent Manuel Fern\'{a}ndez --- \url{https://orcid.org/0000-0002-8035-839X}. \par
\noindent Ignacio Ojeda --- \url{https://orcid.org/0000-0003-3173-5934}.

\bibliographystyle{ws-ijbc}


\begin{thebibliography}{36}
\newcommand{\enquote}[1]{``#1''}
\providecommand{\natexlab}[1]{#1}
\providecommand{\url}[1]{\texttt{#1}}
\providecommand{\urlprefix}{URL }
\expandafter\ifx\csname urlstyle\endcsname\relax
  \providecommand{\doi}[1]{doi:\discretionary{}{}{}#1}\else
  \providecommand{\doi}{doi:\discretionary{}{}{}\begingroup
  \urlstyle{rm}\Url}\fi

\bibitem[{Antali \& Stepan(2018)}]{doi:10.1137/17M1110328}
Antali, M. \& Stepan, G. [2018] \enquote{Sliding and crossing dynamics in
  extended {F}ilippov systems,} \emph{SIAM J. Appl. Dyn. Syst.} \textbf{17},  823--858,
  \url{https://doi.org/10.1137/17M1110328}.

\bibitem[{Arneodo \emph{et~al.}(1982)Arneodo, Coullet \&
  Tresser}]{Arneodo1982171}
Arneodo, A., Coullet, P. \& Tresser, C. [1982] \enquote{Oscillators with
  chaotic behavior: An illustration of a theorem by {Shil'nikov},}
  \emph{J. Stat. Phys.} \textbf{27},  171--182,
  \url{https://doi.org/10.1007/BF01011745}.

\bibitem[{Barnett \& Cameron(1985)}]{barnett1985introduction}
Barnett, S. \& Cameron, R.~G. [1985] \emph{Introduction to Mathematical Control
  Theory}, 2nd Ed., Oxford {A}pplied {M}athematics and {C}omputing {S}cience {S}eries
  (Clarendon Press, Oxford).

\bibitem[{Blows \& Perko(1994)}]{Blows1994341}
Blows, T.~R. \& Perko, L.~M. [1994] \enquote{Bifurcation of limit cycles from
  centers and separatrix cycles of planar analytic systems,} \emph{SIAM Rev.}
  \textbf{36},  341--376,
  \url{https://doi.org/10.1137/1036094}.

\bibitem[{Bravo \emph{et~al.}(2007)Bravo, Fern\'andez \&
  Tineo}]{BravoFernandezTineo07}
Bravo, J.~L., Fern\'andez, M. \& Tineo, A. [2007] \enquote{Periodic solutions
  of a periodic scalar piecewise ode,} \emph{Commun. Pure Appl. Anal.} \textbf{6},  213--228,
  \url{https://doi.org/10.3934/cpaa.2007.6.213}.

\bibitem[{Bravo \emph{et~al.}(2024)Bravo, Fern\'andez \&
  Ojeda}]{BravoFernandezOjeda}
Bravo, J.~L., Fern\'andez, M. \& Ojeda, I. [2024] \enquote{Hilbert number for a
  family of piecewise nonautonomous equations,} \emph{Qual. Theory Dyn. Syst.}
  \textbf{23}, 119,
  \url{https://doi.org/10.1007/s12346-023-00922-9}.

\bibitem[{Carmona \emph{et~al.}(2002)Carmona, Freire, Ponce \&
  Torres}]{CarmonaEtAl02}
Carmona, V., Freire, E., Ponce, E. \& Torres, F. [2002] \enquote{On simplifying
  and classifying piecewise-linear systems,} \emph{IEEE Trans. Circuits Syst. I: Fundam. Theory Appl.} \textbf{49},  609--620,
  \url{https://doi.org/10.1109/TCSI.2002.1001950}.

\bibitem[{Carmona \emph{et~al.}(2003)Carmona, Freire, Ponce \& Torres}]{CFPT2}
Carmona, V., Freire, E., Ponce, E. \& Torres, F. [2003] \enquote{{\'O}rbitas
  peri\'odicas para sistemas tridimensionales lineales a trozos trizonales
  sim\'etricos,}  \emph{Actas XVIII Congreso de Ecuaciones Diferenciales y
  Aplicaciones / VIII Congreso de Matem\'atica Aplicada}, eds. Jornet, J.~M.,
  L\'opez, C., Oliv\'e, J.~M. \& Ram\'{\i}rez, R. (Sociedad Espa\~nola de
  Matem\'atica Aplicada and Universitat Rovira i Virgili, Tarragona).

\bibitem[{Carmona \emph{et~al.}(2004)Carmona, Freire, Ponce \&
  Torres}]{Carmona200471}
Carmona, V., Freire, E., Ponce, E. \& Torres, F. [2004] \enquote{Invariant
  manifolds of periodic orbits for piecewise linear three-dimensional systems,}
  \emph{IMA J. Appl. Math.} \textbf{69},  71--91.

\bibitem[{Carmona \emph{et~al.}(2005)Carmona, Freire, Ponce, Ros \&
  Torres}]{Carmona20053153}
Carmona, V., Freire, E., Ponce, E., Ros, J. \& Torres, F. [2005] \enquote{Limit
  cycle bifurcation in 3{D} continuous piecewise linear systems with two zones.
  Application to {C}hua's circuit,} \emph{Int. J. Bifurcation Chaos} \textbf{15},  3153--3164,
  \url{https://doi.org/10.1142/S0218127405014027}.

\bibitem[{Carmona \emph{et~al.}(2008)Carmona, Fern\'andez-S\'anchez \&
  Teruel}]{Carmona20081032}
Carmona, V., Fern\'andez-S\'anchez, F. \& Teruel, A.~E. [2008]
  \enquote{Existence of a reversible {T}-point heteroclinic cycle in a
  piecewise linear version of the {M}ichelson system,} \emph{SIAM J. Appl. Dyn. Syst.} \textbf{7},  1032--1048,
  \url{https://doi.org/10.1137/070709542}.

\bibitem[{Carmona \emph{et~al.}(2010)Carmona, Fern\'andez-S\'anchez,
  Garc\'ia-Medina \& Teruel}]{Carmona2010}
Carmona, V., Fern\'andez-S\'anchez, F., Garc\'ia-Medina, E. \& Teruel, A.~E.
  [2010] \enquote{Existence of homoclinic connections in continuous piecewise
  linear systems,} \emph{Chaos} \textbf{20}, 013124,
  \url{https://doi.org/10.1063/1.3339819}.

\bibitem[{Carmona \emph{et~al.}(2013)Carmona, Fern\'andez-Garc\'ia, Freire \&
  Torres}]{Carmona201344}
Carmona, V., Fern\'andez-Garc\'ia, S., Freire, E. \& Torres, F. [2013]
  \enquote{Melnikov theory for a class of planar hybrid systems,} \emph{Phys. D} \textbf{248},  44--54,
  \url{https://doi.org/10.1016/j.physd.2013.01.002}.

\bibitem[{Carmona \& Fern\'andez-S\'anchez(2021)}]{CARMONA2021319}
Carmona, V. \& Fern\'andez-S\'anchez, F. [2021] \enquote{Integral
  characterization for {P}oincar\'e half-maps in planar linear systems,}
  \emph{J. Differential Equations} \textbf{305},  319--346,
  \url{https://doi.org/10.1016/j.jde.2021.10.010}.

\bibitem[Carmona \emph{et~al.}(2021)]{CarmonaEtAl19b}
Carmona, V., Fern\'andez-S\'anchez, F. \& Novaes, D.~D. [2021] \enquote{A new
  simple proof for Lum--Chua's conjecture,} \emph{Nonlinear Anal. Hybrid Syst.}
  \textbf{40}, 100992,
  \url{https://doi.org/10.1016/j.nahs.2020.100992}.

\bibitem[{Carmona \emph{et~al.}(2023{\natexlab{a}})Carmona,
  Fern\'andez-S\'anchez \& Novaes}]{Carmona2022}
Carmona, V., Fern\'andez-S\'anchez, F. \& Novaes, D.~D. [2023{\natexlab{a}}]
  \enquote{Uniform upper bound for the number of limit cycles of planar
  piecewise linear differential systems with two zones separated by a straight
  line,} \emph{Appl. Math. Lett.} \textbf{137},  108501,
  \url{https://doi.org/10.1016/j.aml.2022.108501}.

\bibitem[{Carmona \emph{et~al.}(2023{\natexlab{b}})Carmona,
  Fern\'andez-S\'anchez \& Novaes}]{Caretalpre22}
Carmona, V., Fern\'andez-S\'anchez, F. \& Novaes, D.~D. [2023{\natexlab{b}}]
  \enquote{Uniqueness and stability of limit cycles in planar piecewise linear
  differential systems without sliding region,} \emph{Commun. Nonlinear Sci.
  Numer. Simul.} \textbf{123},  107257,
  \url{https://doi.org/10.1016/j.cnsns.2023.107257}.

\bibitem[{Coll \emph{et~al.}(2008)Coll, Gasull \& Prohens}]{CollGasullProhens}
Coll, B., Gasull, A. \& Prohens, R. [2008] \enquote{Simple non-autonomous
  differential equations with many limit cycles,} \emph{Commun. Appl. Nonlinear Anal.} \textbf{15},  29--34.

\bibitem[{di~Bernardo \emph{et~al.}(2008)di~Bernardo, Champneys, Budd \&
  Kowalczyk}]{Bernardo2008}
di~Bernardo, M., Champneys, A.~R., Budd, C.~J. \& Kowalczyk, P. [2008]
  \enquote{Further applications and extensions,}  \emph{Piecewise-smooth
  Dynamical Systems: Theory and Applications} (Springer, London),
  Chapter~9, pp.~409--458.

\bibitem[{Filippov(1988)}]{Filippov88}
Filippov, A.~F. [1988] \emph{Differential equations with discontinuous
  righthand sides}, Mathematics and its Applications (Soviet Series), Vol.~18
  (Kluwer Academic Publishers Group, Dordrecht).

\bibitem[{Freire \emph{et~al.}(1999)Freire, Ponce \& Ros}]{Freire1999895}
Freire, E., Ponce, E. \& Ros, J. [1999] \enquote{Limit cycle bifurcation
  from center in symmetric piecewise-linear systems,} \emph{Int. J. Bifurcation Chaos}
  \textbf{9},  895--907,
  \url{https://doi.org/10.1142/S0218127499000638}.

\bibitem[{Freire \emph{et~al.}(2012)Freire, Ponce \& Torres}]{FreireEtAl12}
Freire, E., Ponce, E. \& Torres, F. [2012] \enquote{Canonical discontinuous
  planar piecewise linear systems,} \emph{SIAM J. Appl. Dyn. Syst.}
  \textbf{11},  181--211,
  \url{https://doi.org/10.1137/11083928X}.

\bibitem[{Gasull \& Zhao(2023)}]{GasullZhao}
Gasull, A. \& Zhao, Y. [2023] \enquote{Existence of at most two limit cycles
  for some non-autonomous differential equations,} \emph{Commun. Pure Appl. Anal.} \textbf{22},  970--982,
  \url{https://doi.org/10.3934/cpaa.2023016}.

\bibitem[{Kahlert \& Chua(1990)}]{Kahlert1990373}
Kahlert, C. \& Chua, L.~O. [1990] \enquote{A generalized canonical
  piecewise-linear representation,} \emph{IEEE Trans. Circuits Syst.} \textbf{37},  373--383.

\bibitem[{Kahlert \& Chua(1992)}]{Kahlert1992222}
Kahlert, C. \& Chua, L.~O. [1992] \enquote{The complete canonical
  piecewise-linear representation---{P}art {I}: The geometry of the domain
  space,} \emph{IEEE Trans. Circuits Syst. I: Fundam. Theory Appl.} \textbf{39},  222--236.

\bibitem[{Khovanskii(1991)}]{Khovanskii}
Khovanskii, A.~G. [1991] \emph{Fewnomials}, Translations of Mathematical
  Monographs, Vol.~88 (American Mathematical Society, Providence, RI).

\bibitem[Leine \& Nijmeijer(2004)]{Leine2004}
Leine, R.~I. \& Nijmeijer, H. [2004] \emph{Dynamics and Bifurcations of
  Non-Smooth Mechanical Systems}, Lecture Notes in Applied and
  Computational Mechanics, Vol.~18 (Springer, Berlin, Heidelberg),
  \url{https://doi.org/10.1007/978-3-540-44398-8}.

\bibitem[Lins Neto(1980)]{Neto1980}
Lins Neto, A. [1980] \enquote{On the number of solutions of the equation
  $\frac{dx}{dt}=\sum_{j=0}^{n}a_j(t)x^j$, $0\leq t\leq1$, for which
  $x(0)=x(1)$,} \emph{Invent. Math.} \textbf{59}, 67--76,
  \url{https://doi.org/10.1007/BF01390315}.

\bibitem[{Massera(1950)}]{Mss}
Massera, J.~L. [1950] \enquote{The existence of periodic solutions of systems
  of differential equations,} \emph{Duke Math. J.} \textbf{17},
  457--475.

\bibitem[{Matsumoto(1984)}]{matsumoto}
Matsumoto, T. [1984] \enquote{A chaotic attractor from {C}hua's circuit,}
  \emph{IEEE Trans. Circuits Syst.} \textbf{31},  1055--1058.

\bibitem[{Novaes \& Silva(2024)}]{Novaes20241486}
Novaes, D.~D. \& Silva, L. V. M.~F. [2024] \enquote{A {M}elnikov analysis on a
  family of second order discontinuous differential equations,} \emph{Sao Paulo J. Math. Sci.} \textbf{18},  1486--1504,
  \url{https://doi.org/10.1007/s40863-024-00436-5}.

\bibitem[{Penalva \emph{et~al.}(2022)Penalva, Desroches, Teruel \&
  Vich}]{10.1063/5.0101778}
Penalva, J., Desroches, M., Teruel, A.~E. \& Vich, C. [2022] \enquote{Slow
  passage through a {H}opf-like bifurcation in piecewise linear systems:
  Application to elliptic bursting,} \emph{Chaos} \textbf{32},  123109,
  \url{https://doi.org/10.1063/5.0101778}.

\bibitem[{Pogorilyi \emph{et~al.}(2014)Pogorilyi, Trivailo \&
  Jazar}]{PogorilyiTrivailoJazar+2014+189+196}
Pogorilyi, O., Trivailo, P.~M. \& Jazar, R.~N. [2014] \enquote{On the piecewise
  linear exact solution,} \emph{Nonlinear Eng.} \textbf{3},  189--196.

\bibitem[{{Python Software Foundation}(2023)}]{PythonSoftwareFoundation2023}
{Python Software Foundation} [2023] \enquote{Python language reference, version
  3.11.2,} \url{https://www.python.org/downloads/release/python-3112/}.

\bibitem[{Schwartz(1953)}]{6371308}
Schwartz, W.~J. [1953] \enquote{Piecewise linear servomechanisms,}
  \emph{Trans. AIEE, Part II: Appl. Ind.} \textbf{71},  401--405.

\bibitem[{Tresser(1984)}]{Tresser84}
Tresser, C. [1984] \enquote{About some theorems by {L. P. \v{S}il'nikov},}
  \emph{Ann. Inst. H. Poincar\'e Phys. Th\'eor.} \textbf{40},  441--461,
  \url{https://archive.numdam.org/item/AIHPA_1984__40_4_441_0/}.

\end{thebibliography}

\end{document}